\documentclass[11pt]{amsproc}
\usepackage{graphicx, amssymb,amsmath, latexsym,cite,fancyhdr}
\usepackage{amsthm,arydshln}     
\usepackage{times}
\usepackage{tikz}    
\usetikzlibrary{arrows,shapes,positioning,calc}
\usepackage[T1]{fontenc} 
\usepackage[switch]{lineno}
\usepackage[ruled,linesnumbered,norelsize]{algorithm2e}
\usepackage{url}  
\usepackage{colortbl}
\usepackage{longtable}

  \newcommand{\scc}{\cellcolor[gray]{0.75}} 

\newtheoremstyle{mythm}
{6pt}
{6pt}
{\it}
{}
{\bf}
{.}
{.5em}
{}

\newtheoremstyle{mydef}
{6pt}
{6pt}
{}
{}
{\bf}
{.}
{.5em}
{}

\newtheoremstyle{myrem}
{6pt}
{6pt}
{}
{}
{\bf}
{.}
{.5em}
{}

\theoremstyle{mythm}      
\newtheorem{theorem}{Theorem}
\newtheorem{lemma}[theorem]{Lemma}
\newtheorem{proposition}[theorem]{Proposition}
\newtheorem{corollary}[theorem]{Corollary}

\theoremstyle{mydef}      
\newtheorem{definition}[theorem]{Definition}
\newtheorem{example}[theorem]{Example}

\theoremstyle{myrem}
\newtheorem{remark}[theorem]{Remark}
\numberwithin{equation}{section}

\newcommand{\bullit}{\item[$\bullet$]}
\newcommand{\tr}{{\rm tr}}
\newcommand{\ad}{\mathop{\rm ad}}

\newcommand{\Span}{{\rm Span}}
\newcommand{\SL}{{\mathfrak{sl}}}
\newcommand{\gl}{{\mathfrak{gl}}}
\newcommand{\GL}{{\rm GL}}

\newcommand{\diag}{\mathrm{diag}}
\newcommand{\g}{{\mathfrak g}}
\newcommand{\fa}{{\mathfrak a}}
\newcommand{\fb}{{\mathfrak b}}
\newcommand{\fd}{{\mathfrak d}}

\newcommand{\fn}{{\mathfrak n}}

\newcommand{\fc}{{\mathfrak c}}
\newcommand{\fk}{{\mathfrak k}}
\newcommand{\fp}{{\mathfrak p}}

\newcommand{\fu}{{\mathfrak u}}

\newcommand{\fs}{{\mathfrak s}}
\newcommand{\ft}{{\mathfrak t}}
\newcommand{\fv}{{\mathfrak v}}
\newcommand{\fz}{{\mathfrak z}}
\newcommand{\h}{{\mathfrak h}}
\newcommand{\su}{{\mathfrak{su}}}
\newcommand{\so}{{\mathfrak{so}}}

\newcommand{\C}{\mathbb{C}}
\newcommand{\Q}{\mathbb{Q}}

\newcommand{\R}{\mathbb{R}}
\newcommand{\Z}{\mathbb{Z}}
\newcommand{\Aut}{{\rm Aut}}
\newcommand{\my}[2]{\langle #1,#2^\vee\rangle}
\newcommand{\myem}[1]{{\textbf{\textit{#1}}}}

\newcounter{ithmcount}

\newenvironment{items}{
\begin{list}{$\alph{item})$}
{\labelwidth18pt \leftmargin18pt \topsep3pt \itemsep2pt \parsep0pt}}
{\end{list}}

\newenvironment{iprf}{\begin{list}{{\rm
	\alph{ithmcount})}}{\usecounter{ithmcount}\labelwidth-5pt
      \leftmargin0pt \topsep3pt \itemsep1pt \parsep2pt}}{\qedhere\end{list}}

\newenvironment{ithm}{\begin{list}{{\rm \alph{ithmcount})}}{\usecounter{ithmcount}\labelwidth18pt
      \leftmargin18pt \topsep3pt \itemsep1pt \parsep2pt}}{\end{list}}

\allowdisplaybreaks

\begin{document}
 
\vspace*{1.3cm}

\title{Regular subalgebras and nilpotent orbits of real graded Lie algebras}

\author[H.\ Dietrich]{Heiko Dietrich}
 \address{School of Mathematical Sciences, Monash University, VIC 3800, Australia}
 \email{heiko.dietrich@monash.edu}
 \author[P.\ Faccin]{Paolo Faccin}
 \address{Department of Mathematics, University of Trento, Povo (Trento), Italy}
 \email{faccin@science.unitn.it}
 \author[W.\ A.\ de Graaf]{Willem A.\ de Graaf}
 \address{Department of Mathematics, University of Trento, Povo (Trento), Italy}
\email{degraaf@science.unitn.it}
\thanks{Dietrich was supported by an ARC-DECRA Fellowship, project
  DE140100088.}

\begin{abstract} 
For a semisimple Lie algebra over the complex numbers,
  Dynkin (1952) developed an algorithm to classify the regular
  semisimple subalgebras, up to conjugacy by the inner automorphism
  group. For a graded semisimple Lie algebra over the complex numbers,
  Vinberg (1979)  showed that a classification of a certain type of regular
subalgebras (called carrier algebras) yields a classification of the nilpotent 
orbits in a homogeneous component
of that Lie algebra. Here we consider these problems for  (graded)
semisimple Lie algebras over the real numbers. First, we describe an algorithm to
classify the regular semisimple subalgebras of a real
semisimple Lie algebra. This also yields an algorithm for listing, 
up to conjugacy, the carrier algebras in a real graded
semisimple real algebra. We then discuss what needs to be done to obtain a
classification of the nilpotent orbits from that; such classifications
have 
applications in differential geometry and  theoretical physics. Our algorithms
are implemented in the language of the computer algebra system {\sf GAP}, using 
our package {\sf CoReLG}; we report on example computations.
\end{abstract}


\maketitle


\section{Introduction}

\noindent  Let $\g^c$ be a complex semisimple Lie algebra with adjoint group $G^c$. Classifying the semisimple subalgebras of $\g^c$ up to $G^c$-conjugacy is an extensively studied problem, partly motivated by applications in theoretical physics;  
see for example \cite{dGraafSS, dyn0, dyn, logru, minchenko}. In
\cite{dyn, logru}, this classification is split into two parts: the  
construction of  regular semisimple subalgebras, that is, semisimple 
subalgebras normalised by a Cartan subalgebra of $\g^c$, and  the 
construction of semisimple subalgebras not contained in any regular proper 
subalgebra. Dynkin \cite{dyn} presented, among other things, an algorithm 
to list the regular semisimple subalgebras of $\g^c$, up to $G^c$-conjugacy. 
One of the main facts underpinning this algorithm is that two semisimple subalgebras, 
normalised by the same Cartan subalgebra $\h^c$, are $G^c$-conjugate if and only
if their root systems are conjugate under the Weyl group of the root system
of $\g^c$ (with respect to $\h^c$). The situation is more intricate for a real 
semisimple Lie algebra $\g$.
As a consequence, here the aim is usually not to classify the semisimple subalgebras, but to decide whether a given real form $\fa$ of a complex  subalgebra $\fa^c$ of $\g^c=\g\otimes_\R\C$ is contained in $\g$, see
\cite{cornsub1, cornsub2, cornsub3, cornsub4, grafac, komrakov}. 
However, for some classes of subalgebras a classification up to $G$-conjugacy
can be obtained, with $G$  the adjoint group of $\g$. Examples are the subalgebras isomorphic to 
$\SL_2(\R)$, whose classification up to $G$-conjugacy is equivalent to 
classifying the nilpotent orbits in $\g$; the latter can be
performed using the Kostant-Sekiguchi correspondence, see \cite{colmcgov, dg}.
It is the first aim of this paper  to show that also the regular
semisimple subalgebras of $\g$ can be classified up to $G$-conjugacy; we describe
an effective algorithm for this task. The main issue is that, in
general, there exist Cartan subalgebras of $\g$ which are not $G$-conjugate, and
that a given regular semisimple subalgebra can be normalised by several non-conjugate
Cartan subalgebras. To get some order in this situation, we introduce the
notion of ``strong $\h$-regularity''; we show that two strongly $\h$-regular subalgebras are $G$-conjugate if and only
if their root systems are conjugate under the real Weyl group of $\g$ (with
respect to $\h$). 

The second problem motivating this paper is the determination of the 
nilpotent orbits in a homogeneous component of a graded semisimple Lie 
algebra. Over the complex numbers (or, more generally, over an algebraically closed 
field of characteristic 0), the theory of orbits in a graded semisimple Lie 
algebra has been developed by Vinberg
\cite{vinberg,vinberg2,vinberg3}. Let $\g^c=\bigoplus_{i\in\Z_m}
\g_i^c$ be a graded semisimple complex Lie algebra, where we write
$\Z_m=\Z/m\Z$, and $\Z_m=\Z$ when $m=\infty$. The component $\g_0^c$ is a
reductive subalgebra of $\g^c$, and $G_0^c$ is defined as the connected subgroup
of $G^c$ with Lie algebra $\g_0^c$. This group acts on the homogeneous
component $\g_1^c$, and the question is what its orbits are. It turns out that
many constructions regarding the action of $G^c$ on $\g^c$ can be generalised
to this setting. In particular, there exists a Jordan decomposition,
so that the orbits of $G_0^c$ in $\g_1^c$ naturally split into three types:
nilpotent, semisimple, and mixed. Of special interest are nilpotent
orbits: in contrast to semisimple orbits, there exist only finitely
many of them. It is known that every nonzero nilpotent  $e\in \g_1^c$
lies in a homogeneous $\SL_2$-triple $(h,e,f)$, with $h\in \g_0^c$
and $f\in \g_{-1}^c$, and that  $G_0^c$-conjugacy of nilpotent elements in $\g_1^c$ is equivalent
to $G_0^c$-conjugacy of the corresponding homogeneous $\SL_2$-triples. Concerning the
classification of the nilpotent orbits, Vinberg \cite{vinberg2} introduced
a new construction: the support, or carrier algebra, of a nilpotent element.
This is a regular $\Z$-graded semisimple subalgebra $\fc^c\leq \g^c$ with certain
extra properties; here ``$\Z$-graded'' means that $\fc^c = \bigoplus_{i\in \Z}
\fc^c_i$ with $\fc_i^c\subseteq \g^c_{i\bmod m}$, and ``regular'' means that
$\fc$ is normalised by a Cartan subalgebra of $\g_0$. The main point
is that the nilpotent $G_0^c$-orbits in $\g_1^c$ are in one-to-one correspondence with the
$G_0^c$-classes of carrier algebras, so that a classification of the former can be obtained from a 
classification of the latter. This approach  has been used in several instances
to classify nilpotent orbits, see for example \cite{elavin, antelash, gatim}; it has also served as the basis of 
several algorithms to classify the nilpotent orbits in $\g_1^c$, see
\cite{gra15, litt8}.

Here we consider the analogous problem over the real numbers: $\g = \bigoplus_{i\in\Z_m}
\g_i$ is a real graded semisimple Lie algebra, $G_0$ is the connected
Lie subgroup of $G$ with Lie algebra $\g_0$, and we want to classify the
nilpotent $G_0$-orbits in $\g_1$. Again, this is much more complicated
than the complex case. Nevertheless, several 
attempts have been made to develop methods for such a
classification. For example, Djokovi{\'c} \cite{djoTri}
considered a $\Z$-grading of the split real Lie algebra of type $E_8$,
so that $\g_0$ is isomorphic to $\gl(8,\R)$, and $\g_1\cong\wedge^3(\R^8)$ as $\g_0$-modules; this is the kind of problem that
is of interest in differential geometry, see \cite{hitchin}. Djokovi{\'c}
classified the corresponding nilpotent orbits (which in this case form all orbits) 
using an approach based on Galois
cohomology; it is not clear whether this can serve as the basis of a 
more general algorithm.  Van L\^e \cite{hongvanle} devised a general method,
whose main idea is to list the possible homogeneous $\SL_2$-triples; however, 
the main step in her approach uses heavy machinery from computational algebraic 
geometry over the real numbers, so that it is questionable whether it will be possible to implement
this method successfully. The problem of classifying real nilpotent orbits is  also considered in the
physics literature (with applications in supergravity), see for example \cite{F4phys,G2phys}. In \cite{F4phys}, the nilpotent orbits
corresponding to a $\Z_2$-grading of the split Lie algebra of type 
$F_4$ are classified using a method based on listing
$\SL_2$-triples. The main idea is to search for real Cayley triples --
and depends on the unproven assumption that each orbit has a representative lying in such
a triple. Using elements of $G_0$, many such triples are shown to be 
conjugate, and the remaining ones are proven non-conjugate by using several
invariants.

Here we approach the problem of classifying the nilpotent orbits
in a real graded semisimple Lie algebra by first listing the carrier algebras.
For this reason, we formulate the algorithm for listing the regular subalgebras
in the more general context of $\Z$-graded regular subalgebras of a graded real
semisimple Lie algebra; the algorithm for listing the regular subalgebras is
then a straightforward specialisation to the trivial grading. From
this, we devise an algorithm for listing the carrier algebras in a
real graded semisimple Lie algebra up to conjugacy; this is the second
aim of this paper. Unfortunately, it is not immediately straightforward to get a classification of
the nilpotent orbits from this list of carrier algebras: unlike in the complex
case, a given carrier algebra can correspond to more than one orbit,
or to no orbit at all. In order to overcome this difficulty, we present a
number of ad hoc techniques, partly similar to the ones used in \cite{F4phys}: 
given a carrier algebra, we first reduce the set of nilpotent orbits to
which it belongs by using elements of the group $G_0$; the remaining
elements are then shown to be non-conjugate by using suitable invariants.

\subsection{Main results and structure of the
  paper}\label{sec:structure} We use the previous notation. In Section \ref{secSS} we recall relevant notation and concepts for
semisimple Lie algebras over the complex and real numbers. As noted above,
the real Weyl group plays a fundamental role in our algorithms; in Section \ref{sec:realweyl} we describe an algorithm to construct the 
real Weyl group of a real semisimple Lie algebra, relative to a given 
Cartan subalgebra. 
The subalgebra $\g_0$ is not semisimple in general, but
reductive. We discuss a number of well-known properties of reductive
subalgebras in Section \ref{secRSA}; these are needed throughout the paper.
Section \ref{secTSA} is devoted to toral subalgebras of a real semisimple
Lie algebra; for such a subalgebra, we study the two subsets of 
elements having only purely imaginary and only real eigenvalues,
respectively. In Section \ref{secGSLA} we consider real graded
semisimple Lie algebras; we recall some well-known
properties, and present two classes of examples which include many cases
of interest. In Section \ref{sec:list} we discuss our first main
algorithm, namely, an algorithm to list all regular $\Z$-graded semisimple subalgebras,
up to $G_0$-conjugacy. In Section \ref{secRSRS} we consider the
specialisation  to the trivial grading to obtain an algorithm for constructing 
the regular subalgebras of $\g$ up to $G$-conjugacy; we exemplify this
with the  real form EI of $E_6$.  The last three sections are devoted to the problem of classifying the nilpotent
orbits of a real graded semisimple Lie algebra; we follow Vinberg's approach 
and use carrier algebras. First, in Section \ref{secCA}, we generalise
some of Vinberg's constructions to the real case; based on our
algorithm in Section \ref{sec:list}, we obtain an algorithm
to list the real carrier algebras in $\g$, up to $G_0$-conjugacy. 
In Section \ref{secNO} we discuss what needs to be done to get the 
classification of the nilpotent orbits from that list of carrier
algebras.  Finally, Section \ref{secex} reports on example computations; in three examples we elaborate on how our methods
behave in practice: We consider the 3-vectors in dimension 8 (as in \cite{djoTri}), an example from
the physics literature (as in \cite{G2phys}), and the real orbits of
$\mathrm{Spin}_{14}(\R)\times \R^*$ on the $64$-dimensional spinor representation
(in the complex case this representation has, for instance, been considered by
\cite{gattivin}).
On some occasions we report on the runtimes of our implementations.

\subsection{Notation}\label{sec:notation}

We use standard notation and terminology for Lie algebras, which, for
instance, can be found
in the books of Humphreys \cite{hum} and Onishchik
\cite{onishchik}. All Lie algebras are denoted by fraktur symbols, for
example, $\g$, and their multiplication is denoted by a Lie bracket
$[-,-]\colon \g\times\g\to \g$; all considered Lie algebras are
finite-dimensional.  If $\varphi\colon \g\to \mathfrak{gl}(V)$ is a representation with $V$ a finite-dimensional vector space, then the associated trace form is $(x,y)=\tr(\varphi(x)\circ\varphi(y))$. The adjoint representation $\ad_\g$ is defined by $\ad_\g(x)(y)=[x,y]$; its trace form is the Killing form $\kappa_\g(x,y)=\tr(\ad_\g(x)\circ\ad_\g(y))$;  if the Lie algebra  follows from the context, then we simply write $\ad$ and $\kappa$. 

Let $\fv\subseteq \g$ be a subspace  and let  $\fa\leq \g$ be a subalgebra. The normaliser and centraliser of $\fv$ in $\fa$ are
\[\fn_{\fa}(\fv) = \{ x \in \fa \mid [x,\fv] \subseteq \fv\}\quad\text{and}\quad\fz_{\fa}(\fv) = 
\{ x \in \fa \mid [x,\fv]=0\},\]respectively. A real form of a complex Lie algebra $\g^c$ is a real subalgebra $\g\leq \g^c$ with $\g^c\cong \g\otimes_\R \C$, that is, $\g^c=\g\oplus\imath \g$ as real vector spaces; here $\imath\in\C$ is the imaginary unit. The real forms of the simple complex Lie algebras are classified; we use the standard notation of \cite[\S C.3 \& C.4]{knapp02}, cf.\ \cite[Table 5]{onishchik}.

As already done above, we endow symbols denoting algebraic structures  over the complex numbers by a superscript $c$. If this superscript is absent, then, unless otherwise noted, the structure is defined over the reals.  Similarly, if $\fv$ is a real vector space, then  $\fv^c = \fv \otimes_\R \C$  is its complexification. 

If $\g^c$ is a complex semisimple Lie algebra, then its adjoint group $G^c$ is the connected Lie subgroup of the automorphism group $\Aut(\g^c)$ with Lie algebra $\ad \g^c$; it is the group of inner automorphisms of $\g^c$, generated by  all $\exp( \ad x)$ with $x\in \g^c$.  Similarly, the adjoint group $G$ of  a real semisimple Lie algebra $\g$ is the connected Lie subgroup of $\Aut(\g)$ with Lie algebra $\ad \g$; it is generated by all $\exp(\ad x)$ with $x\in \g$, see for example \cite[p.\ 126--127]{helgasson}. If $\g^c =  \g\otimes_\R \C$, then $G = G^c(\R)$ is the normaliser of $\g$ in $G^c$, that is, $G=\{g\in G^c\mid g(\g)=\g\}$ is the set of real points of $G^c$.

\section{Semisimple Lie algebras}\label{secSS}
\noindent In this preliminary section we recall some notions concerning semisimple Lie algebras; throughout, $\g^c$ is a semisimple Lie algebra defined over $\C$.

\subsection{Semisimple  complex Lie algebras}\label{subsec:pre1} 
Let $\h^c\leq \g^c$ be a Cartan subalgebra with corresponding root system  $\Phi$. For a chosen root order denote by $\Delta=\{\alpha_1,\ldots,\alpha_\ell\}$ the associated basis of simple roots. The root space corresponding to  $\alpha\in \Phi$ is $\g^c_\alpha=\{g\in\g^c\mid \forall h\in\h^c\colon [h,g]=\alpha(h)g\}$. For $\alpha,\beta\in \Phi$ define $\my{\beta}{\alpha}=r-q$ where $r$ and $q$ are the largest integers such that $\beta-r\alpha$ and $\beta +q\alpha$ lie in $\Phi$. By \cite[\S 25.2]{hum}, there is a \myem{Chevalley basis} of $\g^c$, that is, a basis $\{h_1,\ldots,h_\ell,x_\alpha\mid \alpha\in\Phi\}$ which satisfies $h_i\in \h^c$, $x_\alpha\in \g^c_\alpha$, and 
\begin{equation}\label{eqCB}
\begin{aligned}{}
[h_i,h_j] &= 0,\qquad & [h_i,x_\alpha] &= \my{\alpha}{\alpha_i}  x_\alpha,\\
[x_\alpha,x_{-\alpha}] &= h_\alpha, & [x_\alpha,x_\beta] &= N_{\alpha,\beta} x_{\alpha+\beta};
\end{aligned}
\end{equation}
here $h_\alpha$ is the unique element in $[\g^c_\alpha,\g^c_{-\alpha}]$ with
$[h_\alpha,x_\alpha]= 2x_\alpha$, and $N_{\alpha,\beta} = \pm (r+1)$  where $r$ is the largest integer with $\alpha-r\beta\in\Phi$.  Note that $h_{\alpha_i} = h_i$ for $1\leq i\leq \ell$, and $x_\gamma=0$  for  $\gamma\notin\Phi$. 

A generating set $\{g_i,x_i,y_i \mid i=1,\ldots,\ell\}$ of $\g^c$ is a \myem{canonical generating set} if it satisfies
\begin{equation*}
\begin{aligned}{}
[g_i,g_j]& = 0,\qquad & [g_i,x_j]& = \my{\alpha_j}{\alpha_i} x_j, \\
[x_i,y_j]& = \delta_{ij} g_i,&[g_i,y_j]& = -\my{\alpha_j}{\alpha_i} y_j,
\end{aligned}
\end{equation*}
with $\delta_{ij}$ the Kronecker delta. Sending one canonical generating set to another uniquely extends to an automorphism of $\g^c$, see  
\cite[Thm IV.3]{jac} or
\cite[(II.21) \& (II.22)]{onishchik}. 
Every Chevalley basis $\{h_1,\ldots,h_\ell,x_\alpha\mid \alpha\in\Phi\}$
contains the canonical generating set $\{h_i,x_i,y_i\mid i=1,\ldots,\ell\}$ with $x_i=x_{\alpha_i}$ and $y_i=x_{-\alpha_i}$.

\subsection{Real Forms}\label{subsec:preReal}
Let $\{h_1,\ldots,h_\ell,x_\alpha\mid\alpha\in\Phi\}$ be a Chevalley basis of $\g^c$.
It is well-known and straightforward to verify that the $\R$-span 
\[ \fu=\Span_\R(\{\imath h_1,\ldots, \imath h_\ell,  (x_\alpha-x_{-\alpha}),\imath(x_\alpha+x_{-\alpha})\mid \alpha\in\Phi^+\}).\]
is a \myem{compact form} of $\g^c$, that is, a real form of $\g^c$ with negative definite Killing form; such a form is unique up to conjugacy, see \cite[Cor.\ p.\ 25]{onishchik}. Using the decomposition $\g^c=\fu\oplus\imath \fu$, the associated \myem{real structure} (or conjugation with respect to $\fu$) is $\tau\colon \g^c\to \g^c$, $x+\imath y\mapsto x-\imath y$, where $x,y\in\fu$.

Let $\theta$ be an automorphism of $\g^c$ of order 2, commuting with 
$\tau$. Then $\theta(\fu)\subseteq \fu$, and we can decompose $\fu =
\fu_+ \oplus \fu_{-}$, where $\fu_\pm$  is the $\pm 1$-eigenspace of the restriction of $\theta$ to $\fu$. Now $\g = \fk\oplus \fp$ with $\fk = \fu_+$ and $\fp = \imath\fu_{-}$ is a real Lie algebra with $\g^c=\g\oplus\imath \g$, hence a real form of $\g^c$. The associated real structure (or conjugation) is $\sigma\colon \g^c\to\g^c$, $x+\imath y\mapsto x-\imath y$, where $x,y\in\g$; the maps $\sigma$, $\tau$, and $\theta$ pairwise commute and $\tau=\theta\circ\sigma$.  It is well-known that every real form of $\g^c$ can be constructed in this way, see \cite{onishchik}. The associated decomposition $\g = \fk\oplus \fp$ is a \myem{Cartan decomposition}; the restriction of $\theta$ to $\g$ is  a \myem{Cartan involution} of $\g$. Note that $\theta$ acts on $\fk$ and $\fp$ by multiplication with $1$ and $-1$, respectively. We note that a real form $\g=\fk\oplus\fp$ is compact if and only if  $\fp=\{0\}$.

\begin{lemma}\label{lemCD}
Let $\g=\fk\oplus\fp$ be as before, with Cartan involution
$\theta$. If $\fa\leq \g$ is a semisimple $\theta$-stable subalgebra,
then $\fa=(\fa\cap \fk)\oplus(\fa\cap\fp)$ is a Cartan decomposition
of $\fa$.
\end{lemma}
\begin{proof}
Write $\theta=\sigma\circ\tau$ where $\sigma$ is the complex conjugation associated with $\g$, and $\tau$ is the compact real structure corresponding to the compact real form $\fu$ of $\g^c$. Recall that $\fk=\fu_+$ and $\fp=\imath\fu_-$ where $\fu_\pm$ is the $\pm1$-eigenspace of $\theta|_\fu$. Since $\fa$ is stable under $\sigma$ and $\theta$, we know $\tau(\fa)=\fa$. Write $\g^c=\fu\oplus \imath\fu$, so that $\fa^c=\fb\oplus \imath\fb$ with $\fb=\fa^c\cap \fu$. In particular, $\fb$ is a  real form of $\fa^c$ with real structure $\tau|_{\fa^c}$. By a theorem of Karpelevich-Mostow (see \cite[Cor.\ 6.1]{onishchik}), every Cartan involution of $\fb$ extends to a Cartan involution of $\fu$. In particular, $\fb$ is in fact a compact real form of $\fa^c$. (Since $\fu$ is a compact real form, the Cartan decomposition of $\fb$ must have a trivial '$\fp$-part', hence $\fb$ is compact as well.)  Clearly, $\theta|_{\fa^c}$ is an automorphism of $\fa^c$ commuting with $\tau|_{\fa^c}$, and $\fb=\fb_+\oplus\fb_-$ where $\fb_\pm$ is the $\pm1$-eigenspace of $\theta|_{\fb}$. Now $\fb_+\oplus \imath \fb_-$ is a real form of $\fa^c$ with Cartan involution $\theta|_{\fa^c}$. Note that $\fb_+=(\fa^c\cap\fu)\cap \fu_+=(\fa^c\cap\fk)=\fa\cap \fk$ and $\fb_-=(\fa^c\cap\fu)\cap \fu_-= \imath(\fa\cap\fp)$, which proves the assertion.
\end{proof}

\subsection{Cartan subalgebras}\label{sec:csa}

Let $\g$ be a real semisimple Lie algebra with adjoint group $G$. Let $\g=\fk\oplus\fp$ be a Cartan decomposition with associated Cartan involution $\theta$. By \cite[Prop.\ 6.59]{knapp02}, every Cartan subalgebra of $\g$ is $G$-conjugate to a 
$\theta$-stable Cartan subalgebra. Moreover, Kostant \cite{kostant} and Sugiura \cite{sugiura}
(using independent methods) have shown that, up to $G$-conjugacy,
there are a finite number of Cartan subalgebras in $\g$. We described in 
\cite{dfg}  how the methods of Sugiura yield an algorithm for constructing, up to $G$-conjugacy, all  $\theta$-stable Cartan subalgebras of $\g$. This algorithm has been implemented in our 
software package {\sf CoReLG} \cite{corelg} for the computer algebra
system {\sf GAP} \cite{gap}. 
 
Let $\Phi$ be the root system of $\g^c$ with respect to $\h^c$, where $\h$ is 
as above. Define
\begin{eqnarray*}N_{G^c}(\h^c) &=& \{ g \in G^c \mid g(\h^c)\subseteq \h^c\},\\
Z_{G^c}(\h^c)& =& \{ g \in G^c \mid g(h)=h \text{ for all } h\in \h^c\}.
\end{eqnarray*}
Let $W$ be the Weyl group of $\Phi$, and view $\Phi$ as  subset of the dual space $(\h^c)^\ast$. For $g\in N_{G^c}(\h^c)$ and $\alpha\in \Phi$ define
$\alpha^{g} = \alpha\circ g^{-1}$; using this definition,  $g(\g_\alpha^c) = \g_{\alpha^{g}}^c$, in particular, $\alpha^{g}\in \Phi$. Hence, every $g\in N_{G^c}(\h^c)$ yields a  map \[\psi_g\colon \Phi\to \Phi,\quad \alpha\mapsto \alpha^{g}.\] 
If $g,h\in N_{G^c}(\h^c)$ then $\psi_{g\circ h}$ maps $\alpha$ to $\alpha\circ h^{-1}\circ g^{-1}$, thus $\psi_{g\circ h}=\psi_g\circ \psi_h$.
The next theorem is \cite[Thm 30.6.5]{tauvelyu}; it allows us to define an action of $W$ on $\h^c$.

\begin{theorem}\label{thm:W}
If $g\in N_{G^c}(\h^c)$, then $\psi_g\in W$. The map $N_{G^c}(\h^c)\to W$, $g
\mapsto \psi_g$ is a surjective group homomorphism with
kernel $Z_{G^c}(\h^c)$. In particular, $W\cong 
N_{G^c}(\h^c)/Z_{G^c}(\h^c)$.
\end{theorem}

\begin{lemma}\label{lem:wh}
If $w\in W$, then $w(h_\alpha) = h_{w(\alpha)}$ for all $\alpha\in \Phi$.
\end{lemma}  

\begin{proof}
Theorem \ref{thm:W} shows that $w=\psi_g$ for some $g\in N_{G^c}(\h^c)$, 
and the action of $w$ on $\h^c$ is defined as $w(h)=g(h)$. If $\alpha\in\Phi$, 
then $g(x_\alpha)\in \g_{w(\alpha)}^c$, hence $g(x_\alpha)=\lambda_\alpha x_{w(\alpha)}$ 
for some $\lambda_\alpha\in \C$. It  follows from \eqref{eqCB} that 
$g(h_\alpha)=\lambda_{\alpha}\lambda_{-\alpha}h_{w(\alpha)}$ and  
$\lambda_{-\alpha}=\lambda_{\alpha}^{-1}$, hence  
$w(h_\alpha)=g(h_\alpha)=h_{w(\alpha)}$.
\end{proof}

\section{Computing the real Weyl group}\label{sec:realweyl} 
\noindent  Let $\g=\fk\oplus\fp$ be as in the previous section, with
Cartan involution $\theta=\tau\circ \sigma$ and $\theta$-stable Cartan
subalgebra $\h$; recall that $\tau$ is a compact real structure. Let
$\Phi$ and $W$ be the root system and Weyl group associated with
$\h^c$; let $\{\alpha_1,\ldots,\alpha_\ell\}$ be a basis of simple
roots and let $\{h_1,\ldots,h_\ell,x_\alpha\mid \alpha\in\Phi\}$ be a
Chevalley basis of $\g^c$. Recall the definition of
$h_\alpha=[x_\alpha,x_{-\alpha}]$. We define $N_{G}(\h)$ and $Z_G(\h)$
as in the complex case, and the \myem{real Weyl group} of $\g$
relative to $\h$ as \[W(\h) = N_G(\h)/Z_G(\h),\] see
\cite[(7.92a)]{knapp02}. It follows from  \cite[(7.93)]{knapp02}
that \[W(\h)\leq W.\]  An algorithm for finding generators of $W(\h)$,
based on \cite[Prop.\ 12.14]{adclou}, is implemented in the {\sc Atlas} software 
\cite{atlasLG}. Here we describe a similar, but also more direct algorithm; it is based 
on the following theorem (see \cite[Prop.\ 5.1]{adams} for a very similar statement). 

\begin{theorem}\label{thm:rW} The real Weyl group is $W(\h)=\{w\in W\mid \exists g\in N_{G^c}(\h^c)\colon g\circ\theta=\theta\circ g\text{ and }w=g|_{\h^c}\}$.
\end{theorem}

\begin{proof}
First, we prove ``$\subseteq$''; let $w\in W(\h)$. If $K$ is the connected Lie subgroup of $G$ with Lie algebra
$\fk$, then $W(\h) = N_K(\h)/Z_K(\h)$ by \cite[(7.92b)]{knapp02}. Thus, there
is $g\in N_K(\h)$ whose restriction to $\h^c$ coincides with $w$. Clearly,
all elements of $K$ commute with $\theta$.

Second, we prove ``$\supseteq$''; let $w\in W$ such that $w=g|_{\h^c}$ for some $g\in N_{G^c}(\h^c)$ with $g\circ\theta=\theta\circ g$.
Consider the compact structures $\tau$ and $\tau' = g\circ\tau\circ g^{-1}$. The corollary to \cite[Prop.\ 3.6]{onishchik} shows that $\tau = \eta\circ \tau'\circ \eta^{-1}$ for some $\eta\in G^c$; in particular, one can choose $\eta = \varphi^{-1/4}$, where $\varphi = (\tau'\circ\tau)^2$ and $\varphi^t=\exp(t\log \varphi)$, $t\in\R$, is a 1-parameter subgroup, see \cite[p.\ 23]{onishchik}. Since $\tau'\circ \tau$ commutes with $\theta$, so do $\varphi$ and $\eta$; the latter follows from the fact that $\varphi$ and $\varphi^t$ have the same eigenvectors, see  \cite[p.\ 23]{onishchik}.

If $\alpha\in\Phi$, then $\theta(h_\alpha) = h_{\alpha\circ\theta}$ and $\sigma(h_\alpha) = h_{-\alpha\circ\theta}$, see \cite[Lem.\ 6]{dfg}, hence  $\tau(h_\alpha) = h_{-\alpha} = -h_\alpha$. By Lemma \ref{lem:wh} we have $g(h_\alpha) = h_{w(\alpha)}$,
implying that $\tau'\circ\tau (h_\alpha) = h_\alpha$. Since $\{h_1,\ldots,h_\ell\}$
with $h_i=h_{\alpha_i}$ is a basis for $\h^c$, we get $\tau'\circ\tau\in Z_{G^c}(\h^c)$, thus $\eta\in Z_{G^c}(\h^c)$.

Now define $\tilde{g} = \eta \circ g\in G^c$, so that $\tilde{g}$ commutes with $\theta$ and with $\tau$; for the latter note that $\tau =
\eta\circ \tau'\circ \eta^{-1} = \eta\circ g\circ\tau\circ g^{-1}\circ \eta^{-1} = \tilde{g}\circ\tau\circ \tilde{g}^{-1}$.
In particular, $\tilde{g}$ commutes with $\sigma=\theta\circ\tau$, which proves $\tilde g(\g)=\g$. Thus, $\tilde g\in G^c(\R)=G$.  Now $\eta\in Z_{G^c}(\h^c)$ implies that $\tilde g\in N_{G}(\h)$ and that the restriction of $\tilde{g}$ to $\h^c$ coincides with the restriction $g|_{\h^c}$, hence with $w\in W$ by the definition of $g$. This proves $w\in W(\h)$.
\end{proof}

If $w\in W$, then $w=\psi_g$ for some $g\in N_{G^c}(\h^c)$, see Theorem \ref{thm:W}, and $w$ acts on $\h^c$ as $g$. Let \[W^\theta=\{w\in W\mid \text{the action of $w$ on $\h^c$ commutes with the restriction  $\theta|_{\h^c}$}\}.\] Theorem \ref{thm:rW} yields $W(\h)\leq W^\theta$. We now consider $w\in W(\h)$ and show how to construct $g\in N_{G^c}(\h^c)$ with $w=\psi_g$. Let $\{x_i,y_i,h_i\mid i=1,\ldots,\ell\}$ be the canonical generating set contained in the Chevalley basis of $\g^c$. Clearly,  $\{x_{w(\alpha_i)},x_{-w(\alpha_i)},h_{w(\alpha_i)}\mid i=1,\ldots,\ell\}$ is also a canonical generating set, and mapping $(x_i,y_i,h_i)$ to $(x_{w(\alpha_i)},x_{-w(\alpha_i)},h_{w(\alpha_i)})$ for all $i$ extends uniquely to an automorphism
\[\eta_w\colon\g^c\to \g^c.\]
By Lemma \ref{lem:wh},
the actions of $\eta_w$ and $w$ on $\h^c$ coincide. 
Thus, $\eta_w^{-1}\circ g$ fixes $\h^c$ pointwise; such an automorphism is inner, cf.\ \cite[\S 2.3]{dfg}, hence $\eta_w$ is inner. 

If $z\in Z_{G^c}(\h)$, then $z(x_\alpha)$ is a multiple of $x_\alpha$; in particular, $z$ is determined by nonzero parameters $\lambda_1,\ldots,\lambda_\ell\in \C$ with $z(x_i) = \lambda_i x_i$ and $z(y_i) = \lambda_i^{-1}y_i$; conversely, for such parameters denote by \[\zeta_0(\lambda_1,\ldots,\lambda_\ell)\in Z_{G^c}(\h)\] the automorphism with $z(x_i) = \lambda_i x_i$, $z(y_i) = \lambda_i^{-1}y_i$, and $z(h_i)=h_i$ for all $i$. In conclusion, we have proved the following corollary.

\begin{corollary} The elements in $N_{G^c}(\h^c)$ whose restriction to $\h^c$ is $w\in W(\h)$ are exactly $\eta_w\circ \zeta_0(\lambda_1,\ldots,\lambda_\ell)$ with nonzero  $\lambda_1,\ldots,\lambda_\ell\in\C$.
\end{corollary}

For each $\alpha\in\Phi$ define scalars $\mu_\alpha,\nu_\alpha\in\C$ by 
\begin{equation}\label{eqsc} \theta(x_\alpha) = \mu_\alpha x_{\alpha\circ\theta}\quad\text{and}\quad \eta_w(x_\alpha) = \nu_\alpha x_{w(\alpha)}.
\end{equation} 
Observe that $\mu_{-\alpha} = \mu_\alpha^{-1}$ and $\nu_{-\alpha} = \nu_{\alpha}^{-1}$, and $\theta(h_\alpha)=h_{\alpha\circ\theta}$. For nonzero 
$\lambda_1,\ldots,\lambda_\ell\in\C$ write $\zeta_0 = 
\zeta_0(\lambda_1,\ldots,\lambda_\ell)$.
For $\alpha= \sum_{i=1}^\ell a_i \alpha_i$ define $\lambda_\alpha = 
\prod_{i=1}^\ell \lambda_i^{a_i}$ and $\text{ht}(\alpha)=\sum_{i=1}^\ell a_i$, the height of $\alpha$. An induction on the height shows that $\zeta_0(x_\alpha) = \lambda_\alpha x_\alpha$. Assume $w\in W^\theta$; we want to decide whether there exist nonzero $\lambda_i$ such that $\eta_w\circ\zeta_0\circ \theta = \theta\circ \eta_w\circ\zeta_0$.  Since $w$ and $\theta$ commute we get that
$\eta_w\circ \zeta_0 \circ\theta (h_i)= 
\theta \circ\eta_w\circ\zeta_0(h_i)$ whatever the $\lambda_i$ are. Secondly,
$\eta_w\circ \zeta_0\circ \theta(x_i) = \theta\circ \eta_w\circ\zeta_0(x_i)$ is equivalent to
\begin{equation}\label{eq:lam}
\lambda_{\alpha_i\circ\theta} \lambda_{\alpha_i}^{-1} = \nu_{\alpha_i}\nu_{\alpha_i\circ\theta}^{-1}
\mu_{\alpha_i}^{-1}\mu_{w(\alpha_i)}.
\end{equation}
Thirdly, $\eta_w\circ\zeta_0\circ \theta(y_i) = \theta \circ\eta_w\circ\zeta_0(y_i)$ is equivalent to \eqref{eq:lam}. In conclusion, the next proposition follows.

\begin{proposition}\label{propSE} 
Let $w\in W^\theta$. Then $w\in W(\h)$ if and only if there are nonzero 
$\lambda_1,\ldots,\lambda_\ell\in\C$ satisfying \eqref{eq:lam}  for all $i$.
\end{proposition}

The existence of a solution satisfying \eqref{eq:lam} can readily be checked using row Hermite normal forms, see  \cite[p.\ 322]{sims}. Note that the equations \eqref{eq:lam} are of the form $\lambda_1^{e_{i,1}}\cdots\lambda_\ell^{e_{i,\ell}} = 
u_i$; let $E=(e_{i,j})_{i,j}$ be the matrix of exponents. The left hand side of \eqref{eq:lam} does not depend on $w$, and $e_{i,1},\ldots,e_{i,\ell}$ can be computed from  $\alpha_i\circ\theta-\alpha_i = \sum_j e_{i,j} \alpha_j$, thus $E$ is determined readily. The right hand side of \eqref{eq:lam} does depend on $w$,
and regarding the computation of $u_i = \nu_{\alpha_i}\nu_{\alpha_i\circ\theta}^{-1}
\mu_{\alpha_i}^{-1}\mu_{w(\alpha_i)}$ we remark the following:
\begin{ithm}
\bullit $\mu_\alpha$ with $\alpha\in\Phi$ can be computed directly by the known  action of $\theta$;
\bullit $\nu_{\alpha_i} = 1$ for all $i$ and 
$\nu_{\alpha+\beta} = \nu_\alpha \nu_\beta \tfrac{N_{w(\alpha),w(\beta)}}{N_{\alpha,\beta}}$,
which gives a recursion formula for the $\nu_\alpha$ with $\alpha\in \Phi$.
\end{ithm}
Let $H=PE$ be the row Hermite normal form of $E$, with $P$ an invertible 
$\ell\times\ell$ matrix over $\Z$. Let $(p_{i,1},\ldots,p_{i,\ell})$ be
the $i$-th row of $P$, and define $v_i = u_1^{p_{i,1}}\ldots u_\ell^{p_{i,\ell}}$. If $(h_{i,1},\ldots,h_{i,\ell})$ is the $i$-th row of $H$, then the system of equations \eqref{eq:lam} is equivalent to $\lambda_1^{h_{i,1}}\ldots \lambda_\ell^{h_{i,\ell}}=v_i$ for $i=1,\ldots,\ell$. Thus, the equations \eqref{eq:lam} for $i=1,\ldots,\ell$ have a solution over $\C$ if and only if $v_i=1$ whenever the $i$-th row of $H$ is zero.

Performing this for all $w\in W^\theta$ yields $W(\h)$. However, using 
a theorem of Vogan gives a more efficient algorithm; we need some notation to  formulate
it. As before, let $\Phi$ be the root system of $\g^c$ with 
respect to $\h^c$. By \cite[\S VI.7]{knapp02}, a root $\alpha\in\Phi$
is  \myem{real} if   $\alpha\circ\theta=-\alpha$; it is \myem{imaginary} if $\alpha\circ\theta=\alpha$. An imaginary root $\alpha$ is \myem{compact} if $\theta(x_\alpha) = x_\alpha$. Let $\Phi_{\rm r}$ and $\Phi_{\rm i}$\label{pagerts} be the subsets of $\Phi$ consisting of real and imaginary roots, respectively. These are sub-root systems of $\Phi$,
and we denote by $W_{\rm r}$ and $W_{\rm i}$ their Weyl groups. Define $$h_{\rm r} = \sum\nolimits_{\alpha\in \Phi_{\rm r}^+} h_\alpha\quad\text{and}\quad h_{\rm i} = \sum\nolimits_{\alpha\in \Phi_{\rm i}^+}h_\alpha,$$
and $\Phi_{\rm c} = \{\alpha \in \Phi \mid \alpha(h_{\rm r}) = \alpha(h_{\rm i}) = 0\}$, which
is also a sub-root system, with Weyl group $W_{\rm c}$. Denote by $W_{{\rm ci}}$ the 
Weyl group of the sub-root system consisting of the compact imaginary roots.

For the following theorem we refer to \cite[Props 3.12 \&
  4.16]{vogan}, see also \cite[\S 12]{adclou}.

\begin{theorem}\label{thmVogan}
We have $W^\theta = (W_{\rm r}\times W_{\rm i})\rtimes W_{\rm c}^\theta$ and 
$W(\h) = (W_{\rm r}\times W_{\rm i}^\R)\rtimes W_{\rm c}^\theta$, where $W_{\rm i}^\R = W_{\rm i}\cap W(\h)$.
Moreover,  $W_{\rm ci}$  is contained in $W_{\rm i}^\R$.
\end{theorem}  
Thus, to compute $W(\h)$, we conclude from Theorem \ref{thmVogan} that it is sufficient to test whether $w\in W(\h)$ for $w$ in a set of coset representatives of $W_{\rm ci}$ in $W_{\rm i}$. We remark that generators of $W_{\rm c}^\theta$ are easily 
computed by algorithms that work for general permutation groups.

\section{Reductive subalgebras}\label{secRSA}
\noindent In this section, unless otherwise defined, $\g$ is a semisimple Lie 
algebra over a field of characteristic 0. Recall that $\g$ is  
\myem{reductive} if its adjoint representation is completely reducible.  
This is the same as saying that $\g$ is the direct sum of its centre and its 
derived subalgebra, see  \cite[\S 6, no.\ 4, Proposition 5]{bou} or 
\cite[Def.\ 20.5.1]{tauvelyu}. By the same proposition 
(or \cite[Prop.\ 20.5.4]{tauvelyu}), a Lie algebra is
reductive if and only if it has a finite dimensional representation with
nondegenerate trace form.  Following  \cite[\S 6, no.\ 6, Def.\ 5]{bou} 
or  \cite[Def.\ 20.5.1]{tauvelyu}, a subalgebra $\fa\leq\g$ is 
\myem{reductive in} $\g$ if the $\fa$-module $\g$ (via the
adjoint representation) is completely
reducible. Since $\fa$ is a submodule of
$\g$, this implies that $\fa$ is reductive.  Every  $x\in\g$ can be
written uniquely as $x=s+n$, where $s,n\in\g$ with $\ad_{\g}(s)$ semisimple, 
$\ad_{\g}(n)$
nilpotent, and $[s,n]=0$; this is the \myem{Jordan decomposition} of $x$, 
see   \cite[Prop.\ 4.6.2]{deGraafBook} or \cite[Thm III.17]{jac};  
$s$ and $n$ are the semisimple and  nilpotent part of $x$.

\begin{lemma}\label{lemRedSA}
Let $\fa$ be a subalgebra of $\g$.
\begin{ithm}
\item The subalgebra $\fa$ is reductive in $\g$ if and only if $\fa$ is 
reductive and $\ad_{\g}(z)$ is semisimple for all $z$ in the centre of $\fa$.
\item If the Killing form
of $\g$ restricted to $\fa$ is nondegenerate and $\fa$ contains
the semisimple and nilpotent parts of its elements,
then $\fa$ is reductive in $\g$.
\item  Let $\fa$ be reductive in $\g$. A subalgebra $\ft\leq \fa$ is a Cartan
subalgebra of $\fa$ if and only if $\ft$ is a maximal abelian
subspace of $\fa$ consisting of semisimple elements of $\g$.
\end{ithm}
\end{lemma}
\begin{proof}
\begin{iprf}
\item This is \cite[\S 6, no.\ 5, Th\'eor\`eme 4]{bou}.
\item This is \cite[Prop.\ 20.5.12]{tauvelyu}. In that book the ground 
field is assumed to be algebraically closed. However, the proof given there 
works over any field, replacing the reference to  
\cite[Prop.\ 20.5.4 (iii)]{tauvelyu} by  
\cite[\S 6, no.\ 4, Prop.\ 5d]{bou}.
\item
Since $\fa$ is reductive, $\fa=\mathfrak{c}\oplus\mathfrak{l}$, where $\mathfrak{l}$ is
semisimple and $\mathfrak{c}$ is the centre of $\fa$. Since
$\fa$ is reductive in $\g$, it follows that $\ad_{\g}(z)$ is semisimple
for each $z\in \mathfrak{c}$. A subspace $\ft\subseteq \fa$ is a Cartan subalgebra
if and only if $\ft=\mathfrak{c}\oplus \hat{\ft}$ where $\hat{\ft}$ is a Cartan 
subalgebra of $\mathfrak{l}$. Furthermore, $\hat{\ft}$ is a Cartan subalgebra
of $\mathfrak{l}$ if and only if it is maximally toral; see, for example,
\cite[Cor.\ 15.3]{hum}.
\end{iprf}
\end{proof}

\begin{lemma}\label{lemThInv}
Let $\g$ be a real semisimple Lie algebra with Cartan involution $\theta$.
\begin{ithm}
\item If $\fa$ is a $\theta$-stable subalgebra of $\g$, then $\fa$ is reductive in $\g$.
\item Let $\fa\leq \g$ be a subalgebra; then $\fa$ is reductive in $\g$ if and only if $\fa^c$ is reductive in $\g^c$.
\end{ithm}
\end{lemma}
\begin{proof}
This is proved in  \cite[Cor.\ 1.1.5.4]{warner} and  \cite[\S 6, no.\ 10]{bou}, respectively.
\end{proof}

\begin{remark}
Let $\g$ be a reductive Lie algebra, and $\h$ a Cartan subalgebra.
Let $\fd$ be the derived subalgebra of $\g$. We note that the real Weyl group 
of $\g$ with respect to $\h$ is the same as the real Weyl group of
$\fd$ with respect to $\fd\cap \h$. So the algorithm described in Section
\ref{sec:realweyl} works also in this case.
\end{remark}

\section{Toral subalgebras}\label{secTSA}
\noindent In this section let $\g$ be a real semisimple Lie algebra.
A subalgebra $\fa\leq \g$ is \myem{toral} if it is abelian and 
$\ad_\g (x)$ is semisimple for all $x\in \fa$. Recall that if $\fa$ is reductive in $\g$, then the Cartan subalgebras of $\fa$ are exactly the maximal toral subalgebras of $\fa$, see Lemma \ref{lemRedSA}; in particular, every Cartan subalgebra of $\fa$ lies in some Cartan subalgebra of $\g$. We now study toral subalgebras and their relation to Cartan decompositions.

\begin{lemma}\label{rr1}
Let $\ft\leq \g$ be a toral subalgebra and denote by $\ft_{\rm r},\ft_{\rm i}\subseteq\ft$ the sets of elements $x\in \ft$ such that
$\ad_\g (x)$ has only real and only purely imaginary eigenvalues, respectively.
Let $\theta$ be a Cartan involution of $\g$. Then the following hold.
\begin{ithm}
\item Both $\ft_{\rm r}$ and $\ft_{\rm i}$ are subspaces of $\ft$.
\item If $\ft$ is $\theta$-stable,  then $\ft=\ft_{\rm i}\oplus \ft_{\rm r}$ is decomposition into the $1$- and $(-1)$-eigenspace of the restriction $\theta|_\ft$.
\end{ithm}
\end{lemma}
\begin{proof}
\begin{iprf}
\item This follows from the fact that $\ad_\g (\ft)$ is simultaneously diagonalisable over $\C$.
\item Let $\ft_{\pm}$ be the $\pm 1$-eigenspace of $\theta$. It follows from  
\cite[Prop.\ 5.1(ii)]{onishchik} that $\ft_+ \subseteq \ft_{\rm i}$ 
and $\ft_-\subseteq
\ft_{\rm r}$. Let $x \in \ft_{\rm i}$ and write $x=a+b$ with  $a\in \ft_+$ and  $b\in \ft_-$;
then, since $\ad_\g(x)$ has purely imaginary eigenvalues, $b$ has to be 0. Thus, $\ft_+=\ft_{\rm i}$, hence $\ft_-=\ft_{\rm r}$.
\end{iprf}
\end{proof}

If $\h\leq\g$ is a $\theta$-stable Cartan subalgebra, then by 
Lemma \ref{rr1}b)
\begin{equation*}
\begin{aligned}
\h\cap \fk &= \{ h\in \h \mid {\ad}_{\g} h \text{ has only purely imaginary eigenvalues}\},\\
\h\cap \fp &= \{ h\in \h \mid {\ad}_{\g} h \text{ has only real eigenvalues}\}.
\end{aligned}
\end{equation*}
The dimension of $\h\cap\fp$ is the \myem{noncompact dimension} of
$\h$. Since $\ad_\g(h)$ and $\ad_\g(g(h))$ have the same eigenvalues
for every $h\in\h$ and $g$ in the adjoint group $G$ of $\g$,  it
follows that the  noncompact dimension is a well-defined concept also for non $\theta$-stable Cartan subalgebras: the noncompact dimension of any Cartan subalgebra $\h'$ is the one of $\h=g(\h')$.  It follows from \cite[Prop.\ 6.61]{knapp02} that all Cartan subalgebras of maximal noncompact dimension are $G$-conjugate. The \myem{real rank}\label{defrr} of $\g$ is the noncompact dimension of a maximally noncompact Cartan subalgebra of $\g$, cf.\ \cite[p.\ 424]{knapp02}. The next definition generalises these concepts for subalgebras reductive in $\g$.

\begin{definition}\label{defRR}
Let $\fa\leq \g$ be reductive in $\g$ and let $\h\leq\fa$ be a Cartan subalgebra. The \myem{noncompact dimension of $\h$} is $\dim \h_{\rm r}$, where $\h_{\rm r}$ is as in Lemma \ref{rr1}. A Cartan subalgebra $\h\leq \fa$ is \myem{maximally noncompact} if the noncompact dimension of $\h$  is as large as possible. The noncompact dimension of such a Cartan subalgebra is called the \myem{real rank} of $\fa$.
\end{definition}
In the remainder of this section, we discuss how to describe and compute the real rank of a subalgebra which is reductive in $\g$. We start with a preliminary result.

\begin{lemma}\label{rr2}
If $x\in \g$ is semisimple, then there exist unique $a,b\in \g$ with 
\begin{ithm}
\item[\rm(1)] $x=a+b$ with both $a$ and $b$ semisimple and $[x,a]=[x,b]=[a,b]=0$,
\item[\rm(2)] $\ad_\g(a)$ has purely imaginary eigenvalues only, $\ad_\g(b)$ has real eigenvalues only,
\item[\rm(3)] if $y\in \g$ with $[x,y]=0$, then $[a,y]=[b,y]=0$.
\end{ithm}
The elements $x_{\rm i}=a$ and $x_{\rm r}=b$ are the \myem{imaginary part} and \myem{real part} of $x$.
\end{lemma}

\begin{proof}
Let $\h$ be a Cartan subalgebra of $\g$ containing $x$. Let $\theta$ be
a Cartan involution of $\g$ stabilising $\h$; this exists by \cite[Prop.\ 6.59]{knapp02}.  Let $\g=\fk\oplus\fp$ be
the corresponding Cartan decomposition. Now define
$a = \tfrac{1}{2}(x+\theta(x))\in\fk$ and $b=\tfrac{1}{2}(x-\theta(x))\in\fp$, so that $x=a+b$. Since $a,b\in\h$,  both $a$ and $b$ are semisimple and commute with $x$; in particular, $0=[a,x]=[a,b]$. By Lemma \ref{rr1}, the adjoints $\ad_\g(a)$ and $\ad_\g(b)$ only have purely imaginary and real eigenvalues, respectively. Note that if $x=a'+b'$ with the same properties, then $a'-a=b-b' \in \fk\cap \fp=\{0\}$, thus $a$ and $b$ are unique. It remains to prove (3). With respect to a Chevalley basis of $\g^c$ (with respect to $\h^c$), it follows that $\ad_{\g^c}(a)$ and $\ad_{\g^c}(b)$ are represented by diagonal matrices $A$ and $B$  with  purely imaginary  and real entries, respectively. In particular, $\ad_{\g^c}(x)=\ad_{\g^c}(a)+ \ad_{\g^c}(b)$ is represented by $A+B$. This implies (3).
\end{proof}

\begin{lemma}\label{rr3}
Let $\ft\leq \g$ be a toral subalgebra and $\theta$ a Cartan involution of $\g$. If $\ft$ is $\theta$-stable,
then it is closed under taking real and imaginary parts. Conversely, if
$\ft$ is closed under taking real and imaginary parts, then there is a
Cartan involution $\theta'$ of $\g$ stabilising $\ft$.
\end{lemma}

\begin{proof}
Suppose $\ft$ is $\theta$-stable, and let $x\in \ft$. Then $x=a+b$,
where $a = \tfrac{1}{2}(x+\theta(x))$ and $b=\tfrac{1}{2}(x-\theta(x))$. 
The proof of Lemma \ref{rr2} shows that $a,b\in\ft$ are the imaginary and real
parts of $x$. To prove the converse, let $\h$ be a Cartan subalgebra of $\g$ containing $\ft$, and let $\theta'$ be
a Cartan involution of $\g$ stabilising $\h$; this exists by
\cite[Prop.\ 6.59]{knapp02} and \cite[Cor.\ 6.19]{knapp02}. Let $x\in
\ft$ and let
$x=a+b$ be the decomposition into imaginary and real parts. Since
$a,b\in \ft\leq \h$, Lemma \ref{rr1} yields, $\theta'(a) = a$ and 
$\theta'(b) = -b$, so $\theta'(x) = a-b\in \ft$.
\end{proof}

The next lemma shows that Definition \ref{defRR} is in line with the definition for semisimple Lie algebras.

\begin{lemma}\label{lemMNCCSA}
Let $\fa\leq\g$ be reductive in  $\g$, and decompose $\fa = \fd \oplus \ft$ where $\fd=[\fa,\fa]$ and $\ft$ is the  centre of $\fa$. 
\begin{ithm}
\item Let $\h$ be a Cartan subalgebra of $\fa$ and decompose $\h = \h_d \oplus \ft$ with $\h_d\leq \fd$. Using the notation of Lemma \ref{rr1}, we have $\h_{\rm r} = (\h_d)_{\rm r} \oplus \ft_{\rm r}$.
\item A Cartan subalgebra  $\h\leq \fa$ is  maximally noncompact if and only if 
$\h_d$ is a maximally noncompact Cartan subalgebra of $\fd$. 
\item Up to conjugacy in its adjoint group, $\fa$ has a unique maximally noncompact Cartan subalgebra.
\end{ithm}
\end{lemma}

\begin{proof}
\begin{iprf}
\item  Lemma \ref{lemRedSA} shows that $\fa$ is reductive, so we can decompose $\fa = \fd \oplus \ft$ and $\h=\h_d\oplus \ft$.
Suppose that $\ft$ is not closed under taking real and imaginary parts, say $x=x_{\rm i}+x_{\rm r}\in\ft$ with $x_{\rm i},x_{\rm r}\notin \ft$. Let $\ft_1$ be the subalgebra
spanned by $\ft$, $x_{\rm i}$, and $x_{\rm r}$. It follows from Lemma \ref{rr2} that $\ft_1$ is toral 
and $[\fd,\ft_1]=0$. Iterating this process, we find a toral subalgebra $\ft'$ containing $\ft$, such that  $[\fd,\ft']=0$ and $\ft'$ is closed under taking real and imaginary parts. Note that  $\fd \cap \ft' = 0$ since $[\fd,\ft']=0$ and $\fd$ has trivial centre. By a theorem of Karpelevich-Mostow (see \cite[Cor.\ 6.1]{onishchik}), every Cartan involution of $\fd$ extends to a Cartan involution of $\g$; thus there exists a Cartan involution $\theta$ of $\g$ which stabilises $\h_d$.  Since $\h\leq \g$ (hence also $\h_d\leq \g$) is toral by Lemma \ref{lemRedSA}, it follows from Lemma \ref{rr3} that $\h_d$ is closed under taking real and imaginary parts.  Now let $h\in \h_{\rm r}$ and write $h=u+v$ with $u\in \h_d$ and $v\in \ft$. Decompose $u=u_{\rm r}+u_{\rm i}$ and $v=v_{\rm r}+v_{\rm i}$, and note that  $u_{\rm r},u_{\rm i}\in \h_d$ and $v_{\rm r},v_{\rm i} \in \ft'$. In particular, $h=(u_{\rm r}+v_{\rm r})+(u_{\rm i}+v_{\rm i})$ with $u_{\rm r}+v_{\rm r} \in (\h_d)_{\rm r}\oplus \ft'_{\rm r}$ and $u_{\rm i}+v_{\rm i} \in (\h_d)_{\rm i}\oplus \ft'_{\rm i}$. But $h\in\h_{\rm r}$, so $u_{\rm i}+v_{\rm i}=0$. Since $\fd \cap \ft' = 0$, we get $u_{\rm i}=v_{\rm i}=0$, hence  $\h_{\rm r} \subseteq (\h_d)_{\rm r} \oplus \ft_{\rm r}$. The other inclusion is obvious.
\item By the proof of Part a), there is a Cartan involution $\theta$ which stabilises $\fd$ and $\h_d$. Let $\fd=\fk'\oplus\fp'$ and $\g=\fk\oplus\fp$ be the corresponding Cartan decompositions; clearly, $\fk'\leq \fk$ and $\fp'\leq \fp$. Recall that $\fa\leq \g$ is reductive in $\g$, hence $\h$ is toral, hence $\h_d$ is toral. It follows from Lemma \ref{rr1} that $\h_d=(\h_d)_{\rm r}\oplus (\h_d)_{\rm i}$ with $(\h_d)_{\rm i}= \h_d \cap \fk'$ and $(\h_d)_{\rm r}= \h_d \cap \fp'$. Part a) yields $\h_{\rm r} = (\h_d \cap \fp') \oplus \ft_{\rm r}$, which shows that $\dim \h_{\rm r}$ is as large as possible if and only if  $\dim (\h_d \cap \fp')$ is as large as possible, if and only if $\h_d\leq \fd$ is a maximally noncompact Cartan subalgebra.
\item This follows from b) and the uniqueness of maximally noncompact Cartan subalgebras in semisimple real Lie algebras, see \cite[Prop.\ 6.61]{knapp02}.
\end{iprf}
\end{proof}

\section{Graded semisimple Lie algebras}\label{secGSLA}
\noindent Let $\g$ be a semisimple real Lie algebra. For a positive
integer $m$ let  $\Z_m$ be the integers modulo $m$; in addition,
define  $\Z_\infty = \Z$. A $\Z_m$-grading of $\g$ with
$m\in\mathbb{N}\cup\{\infty\}$ is a decomposition into subspaces
\[\g = \bigoplus\nolimits_{i\in \Z_m} \g_i\]
such that $[\g_i,\g_j]\subset \g_{i+j}$ for all $i,j\in \Z_m$. This implies that $\g_0\leq\g$
is a subalgebra. As usual, write $\g_i^c = \g_i\otimes_\R \C$, so that $\g^c = \bigoplus_{i\in
\Z_m} \g^c_i$. As before, denote by $\sigma$  the conjugation of $\g^c$ with respect to $\g$. 

\begin{lemma}\label{lemg0}
\begin{ithm}
\item There is a Cartan involution $\theta$ of $\g$ such that 
$\theta(\g_i) =\g_{-i}$ for all $i$. 
\item The subalgebras $\g_0^c$ and $\g_0$ are reductive in $\g^c$ and $\g$ respectively.
\end{ithm}
\end{lemma}

\begin{proof}
Part a)  is \cite[Thm 3.4(2)]{hongvanle}.  If $\theta$ is as in  a), then $\theta(\g_0)=\g_0$, and  Lemma \ref{lemThInv} proves that $\g_0$ and $\g_0^c$ are reductive in $\g$ and $\g^c$, respectively.
\end{proof}

\begin{lemma}\label{wtsp}
Let $\h_0^c$ be a Cartan subalgebra of $\g_0^c$. For each $i$, the weight spaces of 
$\h_0^c$ in $\g_i^c$, of nonzero weight, have dimension $1$.
\end{lemma}

\begin{proof}
If  $m=\infty$, then  $\h_0^c$ is a Cartan subalgebra of $\g^c$, see for example \cite[p.\ 370]{djoZ}, so that weight spaces are just root spaces, and therefore they have dimension 1, see \cite[Prop.\ 8.4]{hum}. If  $m$ is a positive integer and  $\omega\in \C$ is a primitive $m$-th
root of unity, then the linear map $\varphi\colon \g^c\to \g^c$ defined by $\varphi(x) =
\omega^i x$ for $x\in \g^c_i$ is  an automorphism of $\g^c$. Clearly, the $\Z_m$-grading of $\g^c$ coincides with the eigenspace decomposition of $\varphi$; now the assertion follows from \cite[Lem X.5.4(i)]{helgasson}.   
\end{proof}

The following lemma is well-known, see \cite[\S 1.4]{vinberg2}.

\begin{lemma}\label{lem:defelt}
If $\g = \bigoplus_{i\in \Z} \g_i$ is a $\Z$-graded real semisimple
Lie algebra, then  there is a unique \myem{defining element} $h_0\in
\g_0$ such that $\g_i=\{x\in \g \mid [h_0,x]=ix\}$ for all $i$.
\end{lemma}

Here we do not try to classify the $\Z_m$-gradings
on a real semisimple Lie algebra $\g$ (for the case of $\Z$-gradings, see
\cite{djoZ}); instead we describe two standard constructions of 
$\Z_m$-gradings, which yield many interesting examples. In both cases, let $\g$ be the real subalgebra of $\g^c$ generated by a canonical generating set  $\mathcal{B}=\{h_i,x_i,y_i\mid i=1,\ldots,\ell\}$ of $\g^c$; then $\g$ is a \myem{split real form} of $\g^c$, see \cite[p.\ 17]{onishchik}, with complex conjugation $\sigma$ (fixing all $x_i$, $y_i$, and $h_i$), and Cartan involution $\theta$  defined by $\theta(h_i) = -h_i$, $\theta(x_i) = 
-y_i$, and $\theta(y_i) = -x_i$ for all $i$, see \cite[Exam.\ 3.2]{onishchik}. Let $\{h_1,\ldots,h_\ell,x_\alpha\mid \alpha\in\Phi\}$ be a Chevalley basis containing $\mathcal{B}$, with  $x_{\alpha_i} = x_i$ and $x_{-\alpha_i} =y_i$, for all simple roots $\alpha_i$. By construction, all $x_\alpha$ lie in $\g$.

\begin{example}\label{exa:grad}
We use the previous notation.
\begin{iprf}
\item We construct a $\Z$-grading and, for this purpose, define a degree of the roots: for each $\alpha_i$ choose
some  integer $d(\alpha_i)\geq0$; for a positive root $\alpha= \sum_i a_i \alpha_i$ define $d(\alpha) = \sum_i a_i d(\alpha_i)$ and $d(-\alpha) = -d(\alpha)$. Let $\g_0$ be the span of $h_1,\ldots,h_\ell$ along with the $x_\alpha$ satisfying  $d(\alpha) = 0$. For $i\in \Z$ define $\g_i$ as the span of all $x_\alpha$ with $d(\alpha) = i$. Then $\g=\bigoplus_{i\in\Z} \g_i$ is a $\Z$-grading with  $\theta(\g_i) = \g_{-i}$ for every $i$.
\item Now let $m\geq 1$ be an integer and $\omega\in\C$ a primitive $m$-th root of unity. Let $\pi$ be a permutation of $\{1,\ldots,\ell\}$, of 
order 1 or 2, such that $\my{\alpha_i}{\alpha_j}
=\my{\alpha_{\pi(i)}}{\alpha_{\pi(j)}}$ for $1\leq i,j\leq \ell$, that
is,  $\pi$ defines a diagram automorphism and a bijection
$\Phi\to\Phi$, also denoted $\pi$. We require that $m$ is even if
$\pi$ has  order  2. Mapping $(x_i,y_i,h_i)$ to $(x_{\pi(i)},y_{\pi(i)}, h_{\pi(i)})$ for all $i$ defines an automorphism of 
$\g^c$, also denoted $\pi$. If $k_1,\ldots,k_\ell$ are non-negative integers with $k_i = k_{\pi(i)}$ for all $i$, then $(x_i,y_i,h_i)\mapsto (\omega^{k_i} ,\omega^{-k_i}y_i,h_i)$ for all $i$ defines an inner automorphism $\eta$ of $\g^c$ which commutes with $\pi$. Suppose the $k_i$ are chosen such that $\eta$ has order $m$, thus $\varphi=\pi\circ \eta$ is an automorphism of $\g^c$ of order
$m$.

If  $\alpha = \sum_j a_j \alpha_j$, then $\eta(x_\alpha) = \omega^r x_\alpha$ where $r=\sum_j a_j k_j$, and each $x_\alpha$ is an eigenvector of $\eta$. Thus, the eigenspace $V_i$ of $\eta$ with eigenvalue $\omega^i\ne 1$ is spanned by all $x_\alpha$ with $\eta(x_\alpha)=\omega^i x_{\alpha}$; the 1-eigenspace $V_0$ of $\eta$ is spanned by $h_1,\ldots,h_\ell$ and $x_\alpha$ with $\eta(x_\alpha)=x_\alpha$.  This and the definition of $\theta$ imply that $\theta(V_i)=V_{-i}$ for all $i$.  Since $x_\alpha\in\g$ for all roots $\alpha$, we also have that $\sigma(V_i)=V_i$ for all $i$.

Let $\g^c=\g^c_+\oplus\g^c_-$ be the $\pm1$-eigenspace decomposition of $\pi$. By construction, $\pi$ and $\theta$ commute, hence $\theta$ fixes $\g_+^c$ and $\g_-^c$. Also, $\pi(x_\alpha)=x_{\pi(\alpha)}$ for every $\alpha\in\Phi$, and it is easy to see that there exist bases of  $\g^c_+$ and $\g^c_-$ which are fixed by $\sigma$, hence $\sigma(\g_+^c)=\g_+^c$ and $\sigma(\g_-^c)=\g_-^c$.

Let $\g^c_i$ be the eigenspace of $\varphi$ with eigenvalue $\omega^i$ for $i=0,1,\ldots,m-1$. Since $\pi$ and $\eta$ commute, $\eta(\g_\pm^c)=\g_\pm^c$ and $\pi(V_i)=V_i$ for all $i$, hence $V_i=(V_i\cap \g^c_+)\oplus (V_i\cap \g^c_-)$. Clearly, $V_i\cap \g_+^c \leq \g_i^c$ and either  $\g_-^c=\{0\}$, or $m$ is even and $V_i\cap \g_-^c \leq \g_{i+m/2}^c$; note that in the latter case $\omega^{m/2}=-1$. Thus, either $\pi=1$ and $\g_i^c=V_i$, or $\pi$ has order 2, $m$ is even, and \[\g_i^c=(\g_+^c\cap V_i) \oplus (\g_-^c\cap V_{m/2+i}).\] Since $\theta(V_i)=V_{-i}$ for all $i$, and $\g_-^c$ and $\g_+^c$ are fixed by $\theta$, in both cases we have $\theta(\g_i^c)=\g_{-i}^c$ and $\sigma(\g_i^c)=\g_i^c$ for all $i$. In particular, if $\g_i=\g_i^c\cap \g$, then $\g=\bigoplus_{i\in\Z_m} \g_i$ and $\theta(\g_i)=\g_{-i}$ for all $i$.
\end{iprf}
\end{example}

\subsection{A representation associated with the grading}\label{sec:rho}
Let $\g=\bigoplus_{i\in\Z_m} \g_i$ be a real semisimple Lie algebra with 
complexification  $\g^c=\bigoplus_{i\in\Z_m} \g_i^c$. It is customary to associate
a representation of an algebraic group to the grading of $\g^c$, see \cite{vinberg,vinberg2}. Namely, as in Section \ref{sec:notation}, let $G^c$ be the adjoint group of 
$\g^c$, and define $G_0^c$ as the connected
algebraic subgroup of $G^c$ with Lie algebra $\ad_{\g^c}(\g_0^c)$. This group acts on $\g_1^c$, yielding a representation \[\rho^c \colon G_0^c \to \GL(\g_1^c).\]
We note that the differential of this representation is $\mathrm{d}\rho^c \colon \g_0^c \to
\gl(\g_1^c)$, given by $\mathrm{d}\rho^c (x)(y) =[x,y]$. 
In the literature, $G_0^c$ is called a ``$\theta$-group'' and $\rho^c$ a 
``$\theta$-representation'', but we avoid using this terminology here 
as we already use ``$\theta$'' to denote a Cartan involution.

Again, as in Section \ref{sec:notation}, let $G$ be the adjoint group
of $\g$. Note that $G_0^c$ and $\rho^c$ are defined over $\R$, so if we define
$G_0$ to be the Lie subgroup of $G$ with Lie algebra $\ad_\g (\g_0)$, then
$G_0 = G_0^c(\R)=\{g\in G^c\mid g(\g)=\g\}$. Furthermore, we can restrict
$\rho^c$ to $G_0$ to obtain a representation $\rho \colon G_0 \to \GL(\g_1)$\label{refrho}.

\section{Listing semisimple regular $\Z$-graded subalgebras}\label{sec:list}

\noindent Let $\g=\bigoplus_{i\in\Z_m} \g_i$ be a real semisimple Lie algebra with 
complexification  $\g^c=\bigoplus_{i\in\Z_m} \g_i^c$. By Lemma \ref{lemg0}, throughout this section,
we suppose that the Cartan involution $\theta$ of $\g$ satisfies \[\theta(\g_i)
= \g_{-i}\quad\text{for all $i$}.\]
A \myem{$\Z$-graded subalgebra}  $\fs^c$ of $\g^c$ is a subalgebra with $\Z$-grading 
$\fs^c = \bigoplus_{k\in \Z} \fs^c_k$ such that $\fs^c_k \subset \g^c_{k \bmod m}$ for all $k$; here $k\bmod m=k$ if $m=\infty$. It is a  \myem{regular subalgebra} of $\g^c$ if it is normalised by some Cartan subalgebra $\h_0^c$ of $\g_0^c$. If we want to specify the particular Cartan subalgebra, then we say that  $\fs^c$ is $\h^c_0$-regular. We  define the same concepts for subalgebras of $\g$. The following is an easy observation.

\begin{lemma}\label{lemStabFSi}
Let $\fs^c\leq \g^c$ be an $\h_0^c$-regular subalgebra, where $\h_0^c\leq \g_0^c$ is a Cartan subalgebra. If $\h_0^c$ contains the (unique) defining elements $h$ of $\fs^c$, then $[\h_0^c,\fs^c_i]\subseteq \fs_i^c$ for all $i$.
\end{lemma}
\begin{proof}
If $k\in \h_0^c$, then $[k,\fs^c]\subseteq \fs^c$ since $\fs^c$ is $\h_0^c$-regular. Thus,  if $x\in\fs_i^c$, then $[h,[k,x]]=-[k,[x,h]]-[x,[h,k]]=[k,[h,x]]=i[k,x]$, and  $[k,x]\in\fs_i^c$ follows. Hence, $[\h_0^c,\fs_i^c]\subseteq \fs_i^c$ for all $i$.
\end{proof}

 In this section, we describe an algorithm for the following task: Given a semisimple $\Z$-graded regular subalgebra $\fa^c$ of $\g^c$, list,
up to $G_0$-conjugacy, all semisimple $\Z$-graded regular subalgebras 
$\fs$ of $\g$ such that $\fs^c$ is $G_0^c$-conjugate to $\fa^c$. First we introduce some notation. Let $\h_0$ be a Cartan subalgebra of $\g_0$, so $\h_0^c$ is a Cartan subalgebra
of $\g_0^c$. For $\lambda\in (\h_0^c)^\ast$ and $k\in \Z_m$ define
\[\g^c_{k,\lambda} = \{ x \in \g_k^c \mid \forall h\in\h_0^c\colon [h,x] = \lambda(h) x\}.\]
By Lemma \ref{wtsp}, if $\lambda\neq 0$, then $\g^c_{k,\lambda}=\{0\}$
or $\g^c_{k,\lambda}$  has dimension 1. Let  \[P(\g^c)=\{(k,\lambda)\mid k\in\Z_m,\lambda\in(\h_0^c)^\ast,\lambda\neq 0,\g_{k,\lambda}^c\ne \{0\}\}.\]  Let
$\fs$ be an $\h_0$-regular $\Z$-graded semisimple subalgebra of $\g$, and set $\fs^c_{k,\lambda}=\g^c_{k,\lambda}\cap \fs^c$ for all $k$ and $\lambda$.
If $\fs^c_{k,\lambda}\ne\{0\}$, then $\fs^c_{k,\lambda}$ is a \myem{weight space} of $\fs^c$ of \myem{weight} $(k,\lambda)$. 
Since $\fs^c$ is $\h_0^c$-regular, it is the sum of its weight spaces. Let \[P(\fs^c)\subseteq P(\g^c)\] be the set of all weights $(k,\lambda)$ of $\fs^c$ with
$\lambda\neq 0$.  Since $[\fs^c_{k,\lambda},\fs^c_{l,\mu}]\subseteq
\fs^c_{k+l,\lambda+\mu}$, weights are added componentwise:
$(k,\lambda)+(l,\mu) = (k+l,\lambda+\mu)$. If  $\kappa$ is  the
Killing form of $\fs^c$, then $\kappa(\fs_{k,\lambda}^c,\fs^c_{l,\mu})
= 0$ unless $l=-k, \mu=-\lambda$. As $\kappa$ is nondegenerate, we
have \[(k,\lambda)\in P(\fs^c)\iff (-k,-\lambda)\in P(\fs^c).\]
Let $W_0^c = N_{G_0^c}(\h_0^c)/Z_{G_0^c}(\h_0^c)$ be the
Weyl group of $\g_0^c$ relative to $\h_0^c$.
The group $W_0^c$ acts on $P(\g^c)$ as follows: if $w\in W_0^c$ with $g\in N_{G_0^c}(\h_0^c)$ projecting to $w$, then  \[w\cdot (k,\lambda) = (k,\lambda^{g});\] recall that $\lambda^g=\lambda\circ g^{-1}$.  
Note that $W_0^c$ is the Weyl group of the root system of $\g_0^c$
relative to $\h_0^c$, see Theorem \ref{thm:W}; hence,  by Lemma \ref{lem:wh}, we know how $W_0^c$ acts on 
$(\h_0)^*$ without computing a $g\in G_0^c$ for a given $w\in W_0^c$. Similarly, $W_0(\h_0) = 
N_{G_0}(\h_0)/Z_{G_0}(\h_0)$ is the real Weyl group relative to $\h_0$; note that $W_0(\h_0)\leq W_0^c$, see Section \ref{sec:realweyl}.

\begin{proposition}\label{propThStab}
Let $\h_0$ be a $\theta$-stable Cartan subalgebra of $\g_0$, and let
$\fs$ be an $\h_0$-regular $\Z$-graded semisimple subalgebra of
$\g$. Then the following hold. 
\begin{ithm}
\item The algebra $\fs$ is $\theta$-stable.
\item The normaliser $\fn_{\g_0}(\fs) = \{ x \in \g_0 \mid [x,\fs]\subseteq \fs \}$  is reductive in $\g$. 
\end{ithm}
\end{proposition}

\begin{proof}
\begin{iprf}
\item
Recall that $\g=\fk\oplus\fp$ and $\h_0=(\h_0\cap\fk)\oplus (\h_0\cap\fp)$, and note that $\sigma$ leaves $\fs$ and $\fs^c$ invariant. Let $(k,\lambda)$ be a weight of $\fs^c$ with $\lambda\neq 0$, thus $\fs^c_{k,\lambda}=\g^c_{k,\lambda}$, and define the linear map $\mu\colon \h_0^c\to\C$ by 
$$ \mu(h) = \begin{cases} \lambda(h) & \text{ if } h\in \h^c_0\cap \fk^c\\
-\lambda(h) & \text{ if } h\in \h^c_0\cap \fp^c.\end{cases}$$
Note that $\sigma$ maps $\g^c_k$ to itself whereas $\theta(\g^c_k) = \g^c_{-k}$.
Using this and Lemma \ref{rr1}, we see that $\sigma(\fs^c_{k,\lambda}) = \g^c_{k,-\mu}$ and  $\theta(\fs^c_{k,\lambda}) = \g^c_{-k,\mu}$. Since $\sigma(\fs^c)= \fs^c$, we conclude that $(k,-\mu)\in P(\fs^c)$ and  $\g^c_{k,-\mu}=\fs_{k,-\mu}^c$. From what is said above, $(-k,\mu)\in P(\fs^c)$, hence $\theta(\fs^c_{k,\lambda})=\fs_{-k,\mu}^c$. Therefore, $\fs^c$, and hence $\fs$, is $\theta$-stable.
\item Since $\fs$ is $\theta$-stable, the same holds for
  $\fn_{\g_0}(\fs)$, and the assertion follows from  Lemma \ref{lemThInv}; recall that $\theta(\g_i)=\g_{-i}$ for all $i$, hence $\g_0$ is $\theta$-stable by assumption.
\end{iprf}
\end{proof}

\begin{definition} Let $\h_0\leq\g_0$ be a Cartan subalgebra. An $\h_0$-regular subalgebra $\fs\leq \g$ is \myem{strongly $\h_0$-regular} if $\h_0$ is maximally noncompact in $\fn_{\g_0}(\fs)$.
\end{definition}
Proposition \ref{propThStab} shows that  $\g_0$ and $\fn_{\g_0}(\fs)$
both are reductive in $\g$. Thus, by Lemma \ref{lemMNCCSA}, there is, up to conjugacy, a unique maximally noncompact Cartan subalgebra of $\fn_{\g_0}(\fs)$.

We end this section with another useful result on $\Z$-graded semisimple subalgebras. 

\begin{lemma}\label{lem:ecen1}
If $\fs\leq \g$ is a $\Z$-graded semisimple subalgebra, then $\fz_{\g_0}(\fs)$ is reductive in $\g$.
\end{lemma}

\begin{proof}
We first show that $C=\fz_{\g}(\fs)$ is reductive in $g$. Let $\kappa$ be the Killing form of $\g$; we consider $\g$ as an $\fs$-module. By Weyl's
  Theorem, see \cite[Thm 4.4.6]{deGraafBook}, the submodule
  $C$ has a complement, say $\g=C\oplus
  U$. Similarly, $[\fs,U]$ is a submodule of $U$ and there is a
  submodule $V$ with  $U=[\fs,U]\oplus V$, hence
  $[\fs,V]=0$. Now $V\subseteq U\cap C=\{0\}$ proves
  $[\fs,U]=U$, thus every $v\in U$ has the form $v=[u,v']$ with
  $v'\in U$ and $u\in\fs$. If $y\in C$, then
  $\kappa(y,v)=\kappa(y,[u,v'])=\kappa([y,u],v')=0$, thus $\kappa|_{C\times U}=0$. Since $\kappa$ is nondegenerate, this implies that  the
  restriction of $\kappa$ to $C$ must be  nondegenerate. Let $g\in C$ with Jordan decomposition $g=s+n$ in $\g$.  Since $g$ centralises $\fs$, so do $s$ and $n$, see \cite[Prop.\ A.2.6]{deGraafBook}, thus $s,n\in C$. It follows from  Lemma \ref{lemRedSA} that $C$ is reductive in $\g$. 

The next step is to show that $C_0=\fz_{\g_0}(\fs)$ is reductive in $\g$. For $i\geq 0$ define $C_i=C\cap \g_i$, 
and $D=\bigoplus_{i\in\Z_m} C_i$; clearly, $D\leq C$. Write $c\in C$ as $c=\bigoplus_{i\in\Z_m} c_i$ with $c_i\in\g_i$. Since $c\in C$, we have $0=[c,s]=\bigoplus_{i\in\Z_m}[c_i,s]$ for all $s\in\fs_j$, and $[c_i,s]\in\g_{i+j\bmod m}$ implies that $[c_i,s]=0$ for all $i$. Since this holds for all $s\in \fs_j$, it follows that  $c_i\in C_i$. We get $C=D$, and  $C$ is a reductive subalgebra of $\g$ with the inherited $\Z_m$-grading. As noted above, the restriction of $\kappa$ to $C$ is 
nondegenerate. If $a\in C_0$ and $b\in C_i$, then  $\ad_{\g}(a)\circ\ad_{\g} (b)$ 
maps $\g_k$ into $\g_{k+i}$, thus $\kappa(a,b)=0$; this implies that the 
restriction of $\kappa$ to $C_0$ is nondegenerate. Let $\varphi$ be the automorphism of $\g^c$ associated with the grading; in particular, note that $\g_0=\{x\in \g\mid \varphi(x)=x\}$. If $g\in C_0$ has Jordan decomposition $g=n+s$ in $\g$, then $n+s=g=\varphi(g)=\varphi(n)+\varphi(s)$, and the uniqueness of Jordan decomposition proves $n=\varphi(n)$ and $s=\varphi(s)$, hence $n,s\in \g_0$. As above, $n,s\in C$, thus $n,s\in C_0$, and Lemma \ref{lemRedSA} proves the assertion.
\end{proof}

\subsection{Computing strongly regular $\Z$-graded subalgebras} We are ready to state some algorithms based on the previous results. Throughout this section, we continue with the assumptions that  $\theta$ is a Cartan involution with $\theta(\g_i) = \g_{-i}$ for all $i$; let $\g=\fk\oplus\fp$ be the corresponding Cartan decomposition, and let $\h_0$ be a $\theta$-stable Cartan subalgebra of $\g_0$.

The next proposition yields an algorithm to decide whether two $\Z$-graded semisimple strongly $\h_0$-regular subalgebras are $G_0$-conjugate.

\begin{proposition}\label{strongconj}
Let $\h_0$ be a $\theta$-stable Cartan subalgebra of $\g_0$ with real Weyl group $W_0(\h_0)$. Let $\fs$ and $\fs'$ be $\Z$-graded semisimple strongly $\h_0$-regular
subalgebras of $\g$. Then $\fs$ and $\fs'$ are $G_0$-conjugate if and only if
$P(\fs^c)$ and $P((\fs')^c)$ are in the same $W_0(\h_0)$-orbit.
\end{proposition}

\begin{proof}
Suppose $a(\fs)=\fs'$ for some $a\in G_0$. First, we show  that $g(\fs) = \fs'$ for some $g\in N_{G_0}(\h_0)$.  To prove this claim, note that $\h_0$ and $a(\h_0)$ both are maximally noncompact Cartan subalgebras of $\fn_{\g_0}(\fs')$. Thus,  $ba(\h_0) = \h_0$ for some $b\in N_{G_0}(\fs')$, and $g=ba\in N_{G_0}(\h_0)$ satisfies  $g(\fs) = \fs'$. Recall that $W_0(\h_0)=N_{G_0}(\h_0)/Z_{G_0}(\h_0)$, and let $w=gZ_{G_0}(\h_0)$. If $(k,\lambda)\in P(\fs^c)$, then $g(\g^c_{k,\lambda}) = \g^c_{k,\lambda^{g}}$, which shows that  $w\cdot(k,\lambda)=(k,\lambda^{g})\in P((\fs')^c)$; in particular,  $P((\fs')^c)=w\cdot P(\fs^c)$, and $P(\fs^c)$ and $P((\fs')^c)$ are conjugate under $W_0(\h_0)$.

Conversely, let  $w\cdot P(\fs^c) =
P((\fs')^c)$ for some  $w\in W_0(\h_0)$; write $w=gZ_{G_0}(\h_0)$ with $g\in N_{G_0}(\h_0)$. Now $w\cdot(k,\lambda) = (k,\lambda^{g})$ and  $g(\fs^c_{k,\lambda}) = (\fs')^c_{k,\lambda^{g}}$ for every weight $(k,\lambda)$, which proves that $g(\fs) = \fs'$. 
\end{proof}

The following proposition yields an algorithm to decide whether an $\h_0$-regular semisimple $\Z$-graded subalgebra is strongly $\h_0$-regular.

\begin{proposition}\label{propSR}
Let $\fs$ be an $\h_0$-regular semisimple $\Z$-graded subalgebra of $\g=\fp\oplus \fk$. Then $\fs$ is strongly $\h_0$-regular if and only if $\fz_{\g_0}(\h_0\cap  \fp)\cap \fn_{\g_0}(\fs)\cap  \fp = \h_0\cap  \fp$.
\end{proposition}

\begin{proof}
Write $\fn_0 = \fn_{\g_0}(\fs)$. It follows from Proposition \ref{propThStab} that  $\fn_0$ is $\theta$-stable and reductive in $\g$. Thus, $\fn_0 = 
\fb\oplus \fc$, where $\fb$ is the derived subalgebra and $\fc$ is the
centre of $\fn_0$; both $\fb$ and $\fc$ are $\theta$-stable, hence we can decompose\[\fn_0=(\fb\cap\fk)\oplus(\fb\cap\fp)\oplus(\fc\cap\fk)\oplus(\fc\cap\fp).\]Every Cartan subalgebra of $\fn_0$ is the direct sum of $\fc$ and a 
Cartan subalgebra of $\fb$;  a maximally noncompact Cartan  subalgebra of $\fn_0$ is the direct sum of $\fc$ and a maximally noncompact Cartan subalgebra of $\fb$. Since $\h_0$ is $\theta$-stable, we can write 
\[\h_0 = (\h_0\cap\fb\cap\fk)\oplus(\h_0\cap\fb\cap\fp)\oplus (\fc\cap\fk)\oplus(\fc\cap\fk);\]note that $\h_b = \fb \cap \h_0$ is a Cartan subalgebra of $\fb$. Since $\fb$ is  $\theta$-stable, it follows from Lemma \ref{lemCD} that $\fb = (\fb\cap\fk) \oplus (\fb\cap \fp)$ is a Cartan decomposition of $\fb$.    Clearly, $\fs$ is strongly $\h_0$-regular if and only if $\h_b$ is a maximally noncompact Cartan subalgebra of $\fb$. By \cite[Prop.\ 6.47]{knapp02} and the remarks in \cite[p.\ 386]{knapp02}, the latter holds if and only if 
$\h_b\cap \fp$ is a maximal abelian subspace of $\fb\cap \fp$. This is the
same as saying that the centraliser of $\h_b\cap \fp$ in $\fb\cap \fp$ is
equal to $\h_b\cap \fp$, which is equivalent to $\fz_{\fn_0}(\fv)\cap \fp = \fv$ with $\fv = \h_0\cap  \fp$.
\end{proof}

To  state the main algorithm of this section, we need one more piece of notation. As before,  let $\h_0$ be a $\theta$-stable Cartan
subalgebra of $\g_0$, and $\theta(\g_i)=\g_{-i}$ for all $i$. Let $W_0^c$ be the 
Weyl group of $\g_0^c$ relative to $\h_0^c$. Let $\fs^c$ be an $\h_0^c$-regular 
$\Z$-graded semisimple subalgebra of $\g^c$. For $w\in W_0$ denote by \[w\cdot \fs^c\]  the 
$\h_0^c$-regular semisimple $\Z$-graded subalgebra of $\g^c$ whose weight system is $w(P(\fs^c))$; then $\fs^c$ and $w\cdot \fs^c$ are called \myem{$W_0$-equivalent}. A semisimple $\Z$-graded subalgebra $\fs^c\leq \g^c$ 
is \myem{$\sigma$-stable}, if each $\fs^c_i$ is $\sigma$-stable.

Algorithm \ref{alg:subalgs} below constructs a list $L$ of strongly $\h_0$-regular semisimple $\Z$-graded subalgebras of $\g$ such that each $\tilde\fs \in L$ is $G_0^c$-conjugate to $\fs^c$, and every strongly $\h_0$-regular semisimple $\Z$-graded 
subalgebra of $\g$ whose complexification is $G_0^c$-conjugate to $\fs^c$ is $G_0$-conjugate to a unique element in $L$.

\begin{algorithm}
\caption{\tt StronglyRegularSubalgebras($\g,\h_0,\fs^c$)}
\footnotesize
\label{alg:subalgs}
\tcc{
$\fs^c$ is a $\h_0^c$-regular semisimple $\Z$-graded subalgebra of $\g^c$, where $\h_0$ is a $\theta$-stable Cartan subalgebra of $\g_0$ and $\theta$ is a Cartan involution with $\theta(\g_i) = \g_{-i}$ for all $i$. Return a list $L$ of strongly $\h_0$-regular semisimple $\Z$-graded 
subalgebras of $\g$ with
\begin{enumerate}
\item the complexification of each  $\tilde\fs \in L$ is $G_0^c$-conjugate to $\fs^c$,
\item  every strongly $\h_0$-regular semisimple $\Z$-graded 
subalgebra of $\g$ whose complex- ification is $G_0^c$-conjugate to $\fs^c$ is $G_0$-conjugate to a unique $\tilde\fs\in L$.
\end{enumerate}
}
\Begin{
compute the root system $\Phi_0$ of $\g_0^c$ with respect to $\h_0^c$, and 
generators of its Weyl group $W_0$\;
compute the real Weyl group $W_0(\h_0)\leq W_0$ (Section \ref{sec:realweyl})\;
compute a set $w_1,\ldots,w_s$ of representatives of the right cosets of 
$W_0(\h_0)$ in $W_0$\;
set $L = \emptyset$ and let $\sigma$ be the real structure defined by $\g$\;
  \For{$1\leq i\leq s$}{
    set $\tilde{\fs}^c = w_i\cdot \fs^c$\;
    \If{\rm$\tilde{\fs}^c$ is $\sigma$-stable} {
      set $\tilde{\fs}=\{x\in \tilde{\fs}^c \mid \sigma(x) = x \}$\;
      \lIf{\rm $\tilde{\fs}$ is strongly $\h_0$-regular (Proposition \ref{propSR})}
    {add $\tilde{\fs}$ to $L$} 
    }
   }
   remove $G_0$-conjugate copies in $L$ (Proposition \ref{strongconj})\;
   \Return{$L$}\;

}
\end{algorithm}

\begin{proposition}\label{prop:6}
Algorithm \ref{alg:subalgs} is correct. 
\end{proposition}

\begin{proof}
Denote by $L$ the output of the algorithm; clearly, $L$ contains no $G_0$-conjugate subalgebras.

Let $\fs'\leq \g$ be a strongly $\h_0$-regular semisimple $\Z$-graded
subalgebra such that $(\fs')^c$ is $G_0^c$-conjugate to $\fs^c$; we
have to show that $\fs'$ is $G_0$-conjugate to an element of $L$. {By
  \cite[Prop.\ 4(2)]{vinberg2}, there exists $w'\in N_{G_0^c}(\h_0^c)$
  with $w'(\fs^c)=(\fs')^c$, thus its projection $w\in W_0^c$ to $W_0^c$ satisfies $w\cdot \fs^c = (\fs')^c$;} write $w=uw_j$ for some $u\in W_0(\h_0)$ and $w_j$ as in Line 4 of the algorithm. In particular, $w_j\cdot \fs^c$ is $W_0^c$-equivalent to $(\fs')^c$. If $i=j$ in the iteration in Line 6, then  $\tilde{\fs}^c = w_j\cdot \fs^c$ is constructed. Since $(\fs')^c$ is $\sigma$-stable and $(\fs')^c$ is $W_0(\h_0)$-equivalent to $\tilde{\fs}^c$, it follows that the latter is $\sigma$-stable as well; the real subalgebra $\tilde{\fs}$ is constructed in Line 9. It also follows that $\tilde{\fs}$ is $W_0(\h_0)$-equivalent  to $\fs'$, thus  $\tilde{\fs}$ is $G_0$-conjugate to $\fs'$ by Proposition \ref{strongconj}. By construction, $\tilde{\fs}$ is strongly $\h_0$-regular as well; it is  added to $L$ in Line 10.
\end{proof}

\begin{proposition}\label{prop:listZgrad}
Let $\g$ and $\theta$ be as before, in particular, $\theta(\g_i)=\g_{-i}$ for all $i$. Let $\fs^c$ be a semisimple $\Z$-graded subalgebra of $\g^c$, and let $\h_0^1,\ldots,\h_0^t$, up to $G_0$-conjugacy, be the $\theta$-stable Cartan subalgebras of $\g_0$. Up to $G_0$-conjugacy, the regular semisimple $\Z$-graded subalgebras of $\g$ that are $G_0^c$-conjugate to $\fs^c$ are $L=L_1\cup\ldots\cup L_t$, where each $L_i$ is  the output of Algorithm \ref{alg:subalgs} with input $(\g,\h_0^i$, $\fs^c)$. 
\end{proposition}

\begin{proof}
Let $\fs'$ be a regular semisimple $\Z$-graded subalgebra of 
$\g$ such that $(\fs')^c$ is $G_0^c$-conjugate to $\fs^c$. 
If $\tilde\h_0$ is a
maximally noncompact Cartan subalgebra of $\fn_{\g_0}(\fs')$, then $g(\tilde\h_0) = \h_0^i$ for some $g\in G_0$ and $i\in\{1,\ldots,t\}$. In particular, $g(\fs')$ is $\h_0^i$-regular. Since
$\tilde\h_0$ is maximally noncompact in $\fn_{\g_0}(\fs')$, we know that
$g(\fs')$ is strongly $\h_0^i$-regular. Note that $g(\fs')$ cannot be strongly $\h_0^j$-regular
with $i\neq j$ since that would imply that $\h_0^i$ and $\h_0^j$ are 
$G_0$-conjugate. From this we conclude that $L$ contains a $G_0$-conjugate of every regular
semisimple $\Z$-graded subalgebra $\fs'$ of $\g$ with $(\fs')^c$ being $G_0^c$-conjugate to $\fs^c$. Moreover, it follows from Proposition \ref{prop:6} that $L$ does not contain $G_0$-conjugate subalgebras.
\end{proof}

We remark that the $\theta$-stable Cartan subalgebras of $[\g_0,\g_0]$ (hence also those of $\g_0$) can, up to conjugacy, be constructed using our algorithms in \cite[\S 4.3]{dfg}.

\section{Regular subalgebras of real simple Lie algebras}\label{secRSRS}

\noindent Let $\g$ be a real simple Lie algebra, with trivial grading.
Let $\h$ be a Cartan subalgebra of $\g$.
Let $\fs^c\leq \g^c$ be an $\h^c$-regular semisimple subalgebra with trivial $\Z$-grading, that is, $\fs^c_0 = \fs^c$. Algorithm 
\ref{alg:subalgs} with input $\h$ and $\fs^c$ returns a list of 
strongly $\h$-regular semisimple subalgebras of $\g$ such that each
strongly $\h$-regular semisimple subalgebra of $\g$ is $G$-conjugate
to exactly one element of the list.

Dynkin \cite[\S 5]{dyn} has given an algorithm for listing the 
$\h^c$-regular semisimple
subalgebras of $\g^c$, up to $G^c$-conjugacy. Furthermore, we can compute a 
list of $\theta$-stable Cartan subalgebras $\h_1,\ldots,\h_r$ of $\g$ such 
that each Cartan subalgebra of $\g$ is $G$-conjugate to exactly one of
them, see \cite[\S 4.3]{dfg}. We perform Algorithm \ref{alg:subalgs} for each $\h_i$ and each 
$\h_i^c$-regular semisimple subalgebra $\fs^c$ of $\g^c$. Taking the union of
all outputs we get a list of all regular semisimple subalgebras of $\g$,
up to $G$-conjugacy; for this, note that every regular semisimple subalgebra $\fs$ of $\g$ is strongly $\h_i$-regular for a unique $i$.

As an example, in Table \ref{tab:regsub} we display the outcome of our computations for the real
form $\g$ of type EI of the simple complex Lie algebra $\g^c$ of type
$E_6$. Recall that $\g=\fk\oplus\fp$ where $\fk$ is simple of type
$C_4$ and $\fp$ has dimension 42. Up to conjugacy, $\g$ has five Cartan
subalgebras $\h_1,\ldots,\h_5$ where $\h_i$ has noncompact dimension
$7-i$. Let $\Phi_i$ denote the root system of $\g^c$ with respect to 
$\h_i^c$. The first set of rows in Table \ref{tab:regsub}, labelled 'real rts',
'im.\ rts', and 'cpt.\ im.\ rts', gives the subsystems of real roots, of
imaginary roots, of compact imaginary roots of $\Phi_i$, respectively
(see Section \ref{sec:realweyl}). The
second set of rows gives the cardinality of the real Weyl group $W(\h_i)$ 
and the index
$[W^c:W(\h_i)]$; subsequently, the runtimes (in seconds) are given for computing $W(\h_i)$ and for constructing all strongly $\h_i$-regular subalgebras of $\g$ 
up to conjugacy; all runtimes have been obtained on a 3.16GHz
machine. The rows below list these regular subalgebras. Up to
$G^c$-conjugacy there are 19 regular subalgebras in $\g^c$; we have 
assigned a number to each of them (without intending any sort of order).
Since $\h_1$ is the split Cartan subalgebra of $\g$, the column of 
$\h_1$ has in each row exactly one real subalgebra, which is the split form
of the corresponding complex regular subalgebra. The columns of $\h_i$
have, on row $j$, the strongly $\h_i$-regular subalgebras which are 
$G_0^c$ equivalent to the $j$-th complex regular subalgebra. On many occasions
there is simply nothing, meaning that there are no such subalgebras, and in
other places there are more than one -- for those we have just added a row
to the table without repeating the label of the complex regular subalgebra.

Concerning the runtimes, we see that the time needed to compute $W(\h_i)$
has no impact on the total time. Also, it is seen that the runtimes
increase sharply if the index $[W^c:W(\h)]$ increases: this is to be expected
as the main iteration in Algorithm \ref{alg:subalgs} runs over a set of
coset representatives.

\begin{table}[ht]
\caption{Regular subalgebras of the real form of type EI of $E_6$}\label{tab:regsub}
\setlength{\extrarowheight}{0.6ex} 
\centering \small
\begin{tabular}{|l|c|c|c|c|c|}
\hline
\multicolumn{1}{|c|}{\scc CSA} & \scc $\h_1$ &\scc $\h_2$ &\scc $\h_3$ &\scc $\h_4$ &\scc $\h_5$\\ 
\hline
real rts & $E_6$ & $A_5$ & $A_3$ & $A_1$ & 0\\
im.\ rts & 0 & $A_1$ & $2A_1$ & $3A_1$ & $D_4$\\
cpt.\ im.\ rts & 0 & 0 & 0 & 0 & $4A_1$\\
\hline
$|W(\h)|$ & 51840 & 1440 & 192 & 96 & 384\\
$[W^c:W(\h)]$ & 1 & 36 & 270 & 540 & 135\\\hline
time $W(\h)$ & 0.03 & 0.17 & 0.72 & 0.99 & 1.61 \\
total time & 28 & 459 & 3191 & 10949 & 3864\\
\hline
\multicolumn{6}{|c|}{\scc\small\slshape Regular subalgebras of EI} \\\hline
1 & $\SL_2(\R)$ & & & & $\su(2)$\\\cdashline{1-6}
2 & $2\SL_2(\R)$ & & $\SL_2(\C)$ & & $2\su(2)$\\\cdashline{1-6}
3 & $\SL_2(\R)\oplus \SL_3(\R)$ & & & &\\\cdashline{1-6}
4 & $\SL_5(\R)$ & & & & \\\cdashline{1-6}
5 & $\so(5,5)$ & & & &\\\cdashline{1-6}
6 & $\SL_2(\R)\oplus \SL_5(\R)$ & & & &\\\cdashline{1-6}
7 & $2\SL_2(\R)\oplus \SL_3(\R)$ & & & $\SL_2(\C)\oplus\su(1,2)$ & \\\cdashline{1-6}
8 & $\SL_2(\R)\oplus 2\SL_3(\R)$ & & & $\SL_2(\R)\oplus\SL_3(\C)$ &
$\su(2)\oplus \SL_3(\C)$\\\cdashline{1-6}
9 & $\SL_2(\R)\oplus\SL_4(\R)$ & & & & $\su(2)\oplus \SL_2(\mathbb{H})$\\\cdashline{1-6}
10 & $3\SL_2(\R)$ & & $\SL_2(\R)\oplus\SL_2(\C)$ & & $\su(2)\oplus\SL_2(\C)$\\
& & & & & $3\su(2)$\\\cdashline{1-6}
11 & $\SL_3(\R)$ & & & $\su(1,2)$ & \\\cdashline{1-6}
12 & $\SL_4(\R)$ & & $\su(2,2)$ & & $\SL_2(\mathbb{H})$\\\cdashline{1-6}
13 & $\SL_6(\R)$ & & & & $\SL_3(\mathbb{H})$\\\cdashline{1-6}
14 & $2\SL_3(\R)$ & & & $\SL_3(\C)$ & \\\cdashline{1-6}
15 & $\so(4,4)$ & & $\so(3,5)$ & & \\\cdashline{1-6}
16 & $\SL_2(\R)\oplus\SL_6(\R)$ & & & & $\su(2)\oplus \SL_3(\mathbb{H})$\\\cdashline{1-6}
17 & $2\SL_2(\R)\oplus \SL_4(\R)$ & & $\SL_2(\C)\oplus \su(2,2)$ & &
$2\su(2)\oplus \SL_2(\mathbb{H})$ \\\cdashline{1-6}
18 & $4\SL_2(\R)$ & & $2\SL_2(\R)\oplus\SL_2(\C)$ & & $2\su(2)\oplus \SL_2(\C)$\\
& & & $2\SL_2(\C)$ & & $4\su(2)$\\\cdashline{1-6}
19 & $3\SL_3(\R)$ & & & $\SL_3(\C)\oplus \su(1,2)$ & \\
\hline
\end{tabular}
\end{table}

\section{Carrier algebras}\label{secCA}
\noindent We continue with our assumption that $\theta$ is a Cartan involution of $\g$ with $\theta(\g_i)=\g_{-i}$ for all $i$.

\begin{definition} A semisimple $\Z$-graded subalgebra $\fs^c$ of $\g^c$ is a 
\myem{carrier algebra} in $\g^c$ if 
\begin{items}
\item[(1)] $\fs^c$ is regular, that is, normalised by some Cartan subalgebra of $\g_0^c$,
\item[(2)] $\fs^c$ is complete, that is, not a proper subalgebra of a reductive
$\Z$-graded regular subalgebra of $\g^c$ of the same rank as $\fs^c$,
\item[(3)] $\fs^c$ is locally flat, that is, $\dim_\C \fs_0^c = \dim_\C \fs_1^c$.
\end{items}
A $\Z$-graded subalgebra $\fs$ of $\g$ is defined to be a carrier algebra in $\g$ if $\fs^c$ is a carrier algebra in $\g^c$. A carrier algebra $\fs$ of $\g$ is \myem{principal} if $\fs_0$ is a torus, that is, if it is a Cartan subalgebra of $\fs$.
\end{definition}
 
Results of Vinberg \cite{vinberg2} relate carrier algebras to nilpotent elements. To describe this, we need more notation and preliminary results. Let $e\in \g_1$ be nilpotent (and nonzero), that is, $\ad_\g (e)$ is nilpotent. A variant of the Jacobson-Morozov Theorem, see \cite[Thm 2.1]{hongvanle}, states that there are $h\in \g_0$ and $f\in \g_{-1}$ such that $[h,e]=2e$,    
$[h,f]=-2f$, and $[e,f]=h$. Such a triple $(h,e,f)$ is  a \myem{homogeneous $\SL_2$-triple}, and $h$ is a \myem{characteristic} of $e$. The next lemma follows from \cite[Thm 2.1]{hongvanle}.

\begin{lemma}\label{lem:sl2con}
Two nonzero nilpotent $e,e'\in \g_1$ lying in homogeneous 
$\SL_2$-triples $(h,e,f)$ and $(h',e',f')$, respectively, are $G_0$-conjugate if and only if  the triples are $G_0$-conjugate.
\end{lemma}

\begin{lemma}\label{lem:ecen2}
Let $(h,e,f)$ be an homogeneous $\SL_2$-triple in $\g$;  let $\fa\leq \g$
be the subalgebra generated by $\{h,e,f\}$, and let $\fv$ be the subspace spanned by $e$. If $\h_z$ is a Cartan subalgebra of $\fz_{\g_0}(\fa)$, then $\Span_\R(h) \oplus \h_z$ is toral, and  a maximal toral subalgebra of
$\fn_{\g_0}(\fv)$.
\end{lemma}

\begin{proof}
By Lemma \ref{lemRedSA} and the definition of toral, it is obvious that $\ft=\Span_\R(h) \oplus \h_z$ is
a toral subalgebra; it remains to show that it is maximal. First, note that
$\fn_{\g_0}(\fv) = \Span_\R(h) \oplus \fc$, where $\fc = \fz_{\g_0} (\fv)$: if $u\in \fn_{\g_0}(\fv)$, then $[u,e]=ce$ for some $c\in\R$, hence $u=c'h+u-c'h$ with $c'=c/2$ and $u-c'h\in \fc$.  If $u\in \fc$, then $[[h,u],e]=[h,[u,e]]+[u,[e,h]]=0$, hence $[h,u]\in \fc$, and $\fc$ has a
basis of eigenvectors of $h$; since $[e,\fc]=0$, it follows from
$\SL_2$-theory that the eigenvalues of $h$ on $C$ are  non-negative
integers.  Write $\fc_i$ for the eigenspace with eigenvalue $i$.  If
$u\in \fc_0$, then $[[f,u],e]=[u,h]=0$, hence $[f,u]\in \fc_0$; this
proves that $\fc_0$ is an $\mathfrak{a}$-module. By $\SL_2$-theory,
$\fc_0$ is a direct sum of 1-dimensional $\mathfrak{a}^c$-modules,
whence  $\fc_0 =  \fz_{\g_0}(\fa)$. Suppose, for a contradiction, that
$\ft'\ne \ft$ is a toral subalgebra of $\fn_{\g_0}(\fv)$ containing
$\ft$. Let $t\in \ft'\setminus \ft$. We may suppose that $t\in \fc$, hence $t = \sum_i t_i$ with $t_i\in \fc_i$. Since $\ft'$ is abelian, $[h,t]=0$, which implies that $t_i=0$ for all $i>0$. Thus, $t=t_0\in \fc_0=\fz_{\g_0}(\fa)$. Thus, the subalgebra of $\fz_{\g_0}(\fa)$ generated by $\h_z$ and $t$ is toral. Since $\h_z$ is a Cartan subalgebra of $\fz_{\g_0}(\fa)$, we must have $t\in\h_z$, a contradiction to $t\notin \ft$.
\end{proof}

We now describe a construction of carrier algebras which in the complex
case is due to Vinberg \cite{vinberg2}. 

\begin{definition}\label{defCA}
Let $e\in \g_1$ be nilpotent with homogeneous $\SL_2$-triple $(h,e,f)$ in $\g$. Let $\fa$ be the subalgebra of $\g$ generated by $\{h,e,f\}$. Let $\h_z$ be a maximally noncompact
Cartan subalgebra of $\fz_{\g_0}(\fa)$, see Lemma \ref{lemMNCCSA}.
By Lemma \ref{lem:ecen2}, the subalgebra $\ft = \Span_\R(h) \oplus \h_z$ is maximal toral in $\fn_{\g_0}(\Span_\R(e))$; 
define $\lambda \colon \ft \to \R$ by $[t,e] =\lambda(t)e$, and  note that $\lambda(h)=2$ and $\lambda(t)=0$ for $t\in \h_z$. For $k\in\Z$ define 
\begin{eqnarray}\label{eqgk}\g_k(\ft,\lambda) = \{x\in \g_k \mid [t,x] = k\lambda(t) x \text{ for all }
t\in \ft\}\quad\text{and}\quad \g(\ft,\lambda) = \bigoplus\nolimits_{k\in \Z} \g_k(\ft,\lambda);
\end{eqnarray}
here $\g_k = \g_{k\bmod m}$ if $m<\infty$. Let  $\fc(e,h,\h_z)$ be
the derived subalgebra of $\g(\ft,\lambda)$.
\end{definition}
 The same construction can be performed in $\g^c$ and, by results of Vinberg \cite{vinberg2}, the resulting algebra $\fc^c(e,h,\h_z^c)$ is a carrier
algebra in $\g^c$. Since $\fc^c(e,h,\h_z^c)=\fc(e,h,\h_z)\otimes_\R\C$, we have the following proposition; in particular, $\g(\ft,\lambda)$ is reductive in $\g$ and $\fc(e,h,\h_z)$ is semisimple.

\begin{proposition}\label{propiscar}
Using the notation of Definition \ref{defCA}, the algebra $\fc(e,h,\h_z)$ is a carrier algebra of $\g$.
\end{proposition}

We call $\fc=\fc(e,h,\h_z)$ a \myem{carrier algebra of $e$}. It
depends on $e$, on the choice of 
characteristic $h$, and on the choice of a maximally noncompact Cartan 
subalgebra
$\h_z$. However, it does not depend on the choice of the third element 
of the $\SL_2$-triple, as that one is uniquely determined by $h$ and $e$, see \cite[Thm 2.1]{hongvanle}. By \cite[Thm 2]{vinberg2}, the element  $e$ is \myem{in general position} in $\fc_1$, that is, $[\fc_0,e] = \fc_1$.

\begin{remark}\label{remGP} Let $\fc$ be a carrier algebra with defining element $h_0\in \fc_0$, see Lemma \ref{lem:defelt}. Let $e'\in \fc_1$ be  in general position. Since $\fc$ is locally-flat and  $[\fc_0,e'] = \fc_1$, it follows that $\fz_{\fc_0}(e')=0$, hence $\fn_{\fc_0}(e') = \Span_\R(h_0)$. Thus, there is a unique $h=2h_0\in\fc_0$ with $[h,e']=2e'$. Moreover, by \cite[Thm 2.1]{hongvanle}, there is a unique homogeneous $\SL_2$-triple $(h,e',f')$ in $\fc$. This proves that every nilpotent element in general position lies in a unique homogeneous $\SL_2$-triple in $\fc$ with  characteristic $h=2h_0$. Thus, if $(h',e',f')$ is a homogeneous $\SL_2$-triple in $\fc$ with $e'$ in general position, then $h'/2$ is the defining element of $\fc$.
\end{remark}

\begin{proposition}\label{propcarconj}
Let $e,e'\in \g_1$ be nilpotent with carrier algebras $\fc(e,h,\h_z)$ and $\fc(e',h',\h_z')$. If $e$ and $e'$ are $G_0$-conjugate, then
$\fc(e,h,\h_z)$ and $\fc(e',h',\h_z')$ are $G_0$-conjugate.
\end{proposition}

\begin{proof}
Let $g\in G_0$ with $g(e)=e'$. By Lemma \ref{lem:sl2con}, we can assume that $g(e) = e'$ and $g(h)=h'$. Now $g(\fc(e,h,\h_z)) = \fc(g(e),g(h),g(\h_z))=\fc(e',h',g(\h_z))$, and we have to show that $\fc(e',h',g(\h_z))$ and $\fc(e',h',\h_z')$ are $G_0$-conjugate. Let $\fa'$ be the subalgebra of $\g$ generated by $\{h',e',f'\}$. By definition, $g(\h_z)$ and $\h_z'$ are  maximally noncompact Cartan subalgebras of $\fz_{\g_0}(\fa')$; by Lemma \ref{lemMNCCSA}, they are conjugate under $Z_{G_0}(\fa')$. Since the elements of the latter group leave $e'$ and $h'$ invariant, the assertion follows.
\end{proof}

By Proposition \ref{propcarconj}, we have a well-defined map $\psi$ from the set of $G_0$-orbits of nilpotent elements in $\g_1$ to the set of $G_0$-conjugacy classes of carrier algebras in $\g$, mapping a nilpotent orbit $G_0\cdot e$ to the $G_0$-conjugacy class of the  carrier algebra $\fc(e,h,\h_z)$.
Similarly, we have a map $\psi^c$ from the set of nilpotent $G_0^c$-orbits in $\g_1^c$ to the set of $G_0^c$-conjugacy classes of carrier algebras in $\g^c$.

In the complex case, $\psi^c$ is a bijection, and the inverse of
$\psi^c$ is obtained by mapping a carrier algebra $\fc^c$ to an
element $e\in\fc^c_1$ in general position: by \cite[Thm 4]{vinberg2}, there exist a characteristic $h$ and
torus $\h_z$ such that $\fc^c = \fc^c(e,h,\h_z)$.  In the real case, $\psi$ is not injective: If $\fc\leq\g$ is a carrier algebra and  $e,e'\in \fc_1$ are in general position, then it is not necessarily 
true that $e$ and $e'$ are $G_0$-conjugate. Moreover, it is not necessarily true that $\fc$ is a carrier of $e$ or $e'$. In general, the map $\psi$ is  not even surjective: it is possible that for a given carrier algebra
$\fc \leq \g$ there is no $e\in \fc_1$ in general position such that $\fc=\fc(e,h,\h_z)$ is a carrier algebra
of $e$.

Algorithms are known to list the $G_0^c$-conjugacy classes of carrier
algebras in $\g^c$, see \cite{gra15,litt8,vinberg2}. In combination with Proposition \ref{prop:listZgrad}, this gives an immediate algorithm to list the $G_0$-conjugacy classes of carrier algebras in $\g$. 
We note that all carrier algebras constructed by this procedure
are strongly $\h_0$-regular for some $\theta$-stable Cartan subalgebra
$\h_0\leq \g_0$, so they are $\theta$-stable by Proposition \ref{propThStab}. We also mention that, using this procedure, the defining element of each such carrier algebra $\fc$ lies in $\h_0$, so that $[\h_0,\fc_i]\subseteq\fc_i$ for all $i$ by Lemma \ref{lemStabFSi}.

\section{Nilpotent orbits}\label{secNO}
\noindent We continue with the previous notation and suppose $\g$ is a
$\Z_m$-graded semisimple Lie algebra with Cartan involution $\theta$
reversing the grading. As shown in the previous section, we have an algorithm to compute
a list $L$ of all carrier algebras in $\g$, up to
$G_0$-conjugacy. Using this algorithm,  each such carrier algebra $\fc$ is strongly $\h_0$-regular for a (known) $\theta$-stable Cartan
subalgebra $\h_0\leq\g_0$ with  $[\h_0,\fc_i]\subseteq \fc_i$ for all
$i$. Since the map
$\psi$ of Section \ref{secCA} is not necessarily injective or surjective,
this does not immediately yield the nilpotent $G_0$-orbits in $\g_1$. In this
section, we discuss what needs to be done to obtain the nilpotent orbits from
the list $L$ of carrier algebras.

A fundamental problem that we encounter here is to decide, for a given carrier
algebra $\fc$ and $e\in \fc_1$ in general position, whether \myem{$\fc$ is a 
carrier algebra of $e$}, that is, whether $\fc=\fc(e,h,\h_z)$ for some $\h_z$; note that $h/2$ must be the unique defining element of $\fc$, see Remark \ref{remGP}.  For that we use the next  result.

\begin{proposition}\label{propcarrk}
Let $\fc$ be a carrier algebra in $\g$ with defining element $h/2$. Let
$e\in \fc_1$ be in 
general position, lying in the homogeneous $\SL_2$-triple $(h,e,f)$.
Let $\fa$ be the subalgebra spanned by $\{h,e,f\}$. 
\begin{ithm}
\item If $\widetilde{\h}\leq\fz_{\g_0}(\fc)$ is a Cartan subalgebra of
  $\fz_{\g_0}(\fc)$, then $\widetilde{\h}$ is a Cartan subalgebra of $\fz_{\g_0}(\fa)$.
\item The subalgebra $\fc$ is a carrier algebra of $e$ if and only if the real ranks of $\fz_{\g_0}(\fc)$ and $\fz_{\g_0}(\fa)$ coincide.
\item Write $\fz_{\g_0}(\fa)=\fd\oplus \ft$, where $\fd$ is the derived 
subalgebra and $\ft$ is the centre. Let $\widetilde{\h}$ be a maximally
noncompact Cartan subalgebra of $\fz_{\g_0}(\fc)$. Then $\fc$ is a carrier
algebra of $e$ if and only if $\widetilde{\h}\cap\fd$ is a maximally
noncompact Cartan subalgebra of $\fd$.
\item Let $e'\in \fc_1$ be in general position. If $e$ and $e'$ are $G_0$-conjugate, then $\fc$ is a carrier algebra of $e$ if and only if it is a carrier algebra of $e'$.
\end{ithm}
\end{proposition}
\begin{proof}
By Remark \ref{remGP}, the triple $(h,e,f)$  exists and is uniquely
defined by $e$. Recall that each of $\fz_{\g_0}(\fa)$, and
$\fz_{\g_0}(\fc)$ is reductive in $\g$, see   Lemma
\ref{lem:ecen1}. We make use of Lemma \ref{lemRedSA}c) and Lemma
\ref{lemMNCCSA};  the main ideas in the proof of a) are
borrowed from \cite{vinberg2}. 
\begin{iprf}
\item Set $\fb = \Span_\R(h) \oplus \widetilde{\h}$, and define the
linear map $\varphi\colon \fb \to \R$ by $\varphi(h) = 1$ and $\varphi(u) = 0$
for $u\in \widetilde{\h}$. Let $\g(\fb,\varphi)$ be defined as in
\eqref{eqgk}. Now \cite[Prop.\ 2]{vinberg2} shows that $\g(\fb,\varphi)=\fs\oplus 
\widetilde{\h}$, where $\fs$ is a complete regular semisimple $\Z$-graded
subalgebra, and $\widetilde{\h}$ is the centre. (In \cite{vinberg2},
the complex case is discussed, but the same result follows for real algebras.) Moreover, $\fs$ contains
$\fc$ and has the same rank as $\fc$. As $\fc$ is a carrier
algebra (and therefore complete), we get $\fs=\fc$. 
Note that $\fz_{\g_0}(\fc)\leq \fz_{\g_0}(\fa)$, hence $\widetilde{\h}$ is 
contained in a Cartan subalgebra  $\h_z$ of $\fz_{\g_0}(\fa)$. It
follows easily from the definition of $\g(\fb,\varphi)$ that
$\fz_{\g_0}(\fb)=\g(\fb,\varphi)\cap \g_0 = \fc_0 \oplus
\widetilde{\h}$. Note that $\h_z$ centralises $\widetilde{\h}$ and $h$, hence $\widetilde{\h}\leq \h_z \leq \fz_{\g_0}(\fb)$.
If  the first inclusion would be proper, then there exists a nonzero
$v\in \h_z\setminus \widetilde{\h}$ with $v\in\fc_0$; now $v\in
\fz_{\g_0}(\fa) \cap \fc_0\leq \fz_{\fc_0}(e)=\{0\}$ yields a contradiction.
(Note that $\fz_{\fc_0}(e)=\{0\}$ since $\fc$ is locally-flat and $e$ is in general position.) This proves that $\h_z = \widetilde{\h}$.  
\item Suppose that $\fc$ is a carrier algebra of $e$, that is, $\fc=\fc(e,h,\h_z)$. Due to the uniqueness of the third element in an $\SL_2$-triple, $\h_z$ must be a maximally noncompact Cartan subalgebra of $\fz_{\g_0}(\fa)$, where $\fa$ is as defined in the proposition.  It follows from the
construction of $\fc(e,h,\h_z)$ that $\h_z\leq \fz_{\g_0}(\fc)\leq \fz_
{\g_0}(\fa)$. Hence $\h_z$ is contained in a Cartan subalgebra $
\widetilde{\h}\leq \fz_{\g_0}(\fc)$. It follows from Part a) that both $
\h_z$ and $\widetilde{\h}$ are Cartan subalgebras of  $\fz_{\g_0}(\fa)
$, hence $\h_z=\widetilde{\h}$, and $\h_z$ is a Cartan subalgebra of $
\fz_{\g_0}(\fc)$. The results in Section \ref{secTSA} imply that $\h_z$
is maximally noncompact in $\fz_{\g_0}(\fc)$ since otherwise $\h_z$
would not be maximally noncompact in $\fz_{\g_0}(\fa)$. So the real
ranks of both algebras coincide. 

For the converse, suppose that $\fz_{\g_0}(\fc)$ and $\fz_{\g_0}(\fa)$
have the same real rank. If $\widetilde{\h}$ is a maximally noncompact
Cartan subalgebra of $\fz_{\g_0}(\fc)$, then, using Part a), it follows
that $\widetilde{\h}$ is a maximally noncompact Cartan subalgebra of $\fz_{\g_0}(\fa)
$ as well. If $\fb$ and $\varphi$ are as in a), then that proof  shows
that $\g(\fb,\varphi) = \fc \oplus \widetilde{\h}$. On the other hand,
by construction of $\fc(e,h,\widetilde{\h})$, the derived subalgebra of
$\g(\fb,\varphi)$ is equal to $\fc(e,h,\widetilde{\h})$. So we get $\fc=
\fc(e,h,\widetilde{\h})$, and $\fc$ is a carrier algebra of $e$. 

\item As shown above, $\fc$ is a carrier algebra of $e$ if and only if $\widetilde\h$ is a maximally noncompact Cartan subalgebra of $\fz_{\g_0}(\fa)$, if and only if  $\widetilde\h\cap \fd$ is a maximally noncompact Cartan subalgebra of $\fd$, see Lemma \ref{lemMNCCSA}b).

\item Suppose $e'=g(e)$ with $g\in G_0$, and note that $e'$ lies in a
  unique $\SL_2$-triple $(h,e',f')$ in $\fc$, see Remark
  \ref{remGP}. It follows from Lemma \ref{lem:sl2con} that $(h,e,f)$
  and $(h,e',f')$ are $Z_{G_0}(h)$-conjugate, thus we can assume that
  $g(h)=h$ and $g(e)=e'$. Let $\fa$ and $\fa'$ be the algebras spanned
  by $\{h,e,f\}$ and $\{h,e',f'\}$, respectively; note that
  $\fa'=g(\fa)$, hence $\fz_{\g_0}(\fa)$ and $\fz_{\g_0}(\fa')$ have
  the same real rank. Now b) proves the assertion.
\end{iprf}
\end{proof}

\begin{remark}\label{rem:iscar}
Using our algorithms in \cite{dfg} for computing a Cartan decomposition and a maximally noncompact Cartan subalgebra of a
real semisimple Lie algebra, Proposition \ref{propcarrk}c) gives an immediate algorithm to decide whether $\fc$ is a carrier of a 
given $e\in \fc_1$. 
\end{remark}

\noindent In summary, an approach for computing (representatives of) the nilpotent $G_0$-orbits in $\g_1$ is as follows. Recall the definition of the representation $\rho \colon G_0 \to \GL(\g_1)$, see Section \ref{sec:rho}.

\begin{items}
\item[(a)] Use the methods of Section \ref{secCA}
to compute the list $L$ of carrier algebras in $\g$, up to $G_0$-conjugacy.
\item[(b)] For each carrier algebra $\fc\in L$ with defining element $h/2$  do the following:
\begin{ithm}
\item[(b1)] Use linear algebra methods to determine the set $\mathcal{G}$ of elements $e\in \fc_1$  in
general position.
\item[(b2)] Use elements of $\rho(G_0)$ to find a {\it small} finite set $\mathcal{G}'\subseteq
\mathcal{G}$, such that every element of $\mathcal{G}$ is $G_0$-conjugate
to at least one element of $\mathcal{G}'$.  The objective is to make
 $\mathcal{G}'$ as small as possible; ideally, we want to reduce $\mathcal{G}$ up to $G_0$-conjugacy.
\item[(b3)] Let  $\mathcal{C}$ be set  of those $e\in\mathcal{G}'$ such that $\fc$ is a carrier algebra of $e$, see Remark \ref{rem:iscar}.
\item[(b4)] Use $G_0$-invariants to show that the elements of $\mathcal{C}$
are {\em not} $G_0$-conjugate.
\end{ithm}
\item[(c)] The union of all sets $\mathcal{C}$ that have been obtained is a complete and irredundant set of $G_0$-orbit representatives of nilpotent elements in $\g_1$, see Proposition \ref{propcarconj}.
\end{items}

This is not an algorithm, but more a structured programme of work:   it is not immediately clear how to 
perform steps (b1), (b2), and (b4). We use the previous notation and  comment on each step below.

\subsection{The set of elements in general position}\label{sec:genpos}
Let $\fc$ be a carrier algebra, and let $\{x_1,\ldots,x_s\}$ and $\{y_1,\ldots,y_s\}$ be bases of $\fc_0$ and $\fc_1$
respectively; recall that $\dim \fc_0=\dim \fc_1$ since $\fc$ is locally-flat. Write $e=c_1y_1+\ldots+c_sy_s\in\fc_1$ and recall that $e$ is in general position if and only if $[\fc_0,e]=\fc_1$, that is, if and only if $\ad_\g(e)$ induces  a bijection $\fc_0\to\fc_1$. For each $y_i$ let $M_i$ be the matrix describing $\ad_\g(y_i)\colon \fc_0\to \fc_1$, $x\mapsto [y_i,x]$. Then $\ad_\g(e)|_{\fc_0}$ is represented by the matrix $M(c_1,\ldots,c_s)=c_1M_1+\ldots+c_s M_s$, which is a bijection if and only if $f(c_1,\ldots,c_s)=\det(M(c_1,\ldots,c_s))$ is non-zero. It follows that the set $\mathcal{G}$ of elements in $\fc_1$ in general position can be described as \[\mathcal{G}=\{c_1y_1+\ldots+c_sy_s\mid f(c_1,\ldots,c_s)\ne 0\}.\]By itself, the condition $f(c_1,\ldots,c_s)\neq 0$ is not very revealing.
However, on many occasions we can factorise $f$, and the factors tend to
give useful information.

\subsection{Finding elements of $\rho(G_0)$}\label{secG}
Let $\mathcal{G}$ be as in the previous section; we want to reduce $\mathcal{G}$ to a {\it small} finite subset $\mathcal{G}'$ such that each $e\in\mathcal{G}$ is $G_0$-conjugate to at least one element in $\mathcal{G}'$. The problem is that, in general, we do not know $\rho(G_0)$. Here we discuss how to construct  elements of $\rho(G_0)$ which often can be used to bring elements in $\mathcal{G}$ into the form  $c_1y_1+\ldots+c_sy_s$ where some of the coefficients $c_i$ are 0, while others are $\pm 1$.

\subsubsection{Nilpotent elements}\label{secsn}
If $x$ is a nilpotent element of $\fc_0$, then $\exp(\ad_\g(x))$ lies in
$G_0$. We know from \cite[Thm 0.23]{knapp02} that  $\rho(\exp(\ad_\g(x))) =
\exp (\text{d}\rho(x) )$, and  $\text{d}\rho(x)$ is the restriction of $\ad_\g(x)$
to $\g_1$. Thus,  $\rho(\exp(\ad_\g(x)))$ stabilises $\fc_1$, and
its restriction to $\fc_1$ is $\exp(y)$, where $y$ is the restriction of
$\ad_\g(x)$ to $\fc_1$. Typically, we can use such elements to show that some coefficients of $e\in\mathcal{G}$ may be
assumed to be 0.

\subsubsection{Semisimple elements}\label{secsss}
By construction, $\fc$ is $\h_0$-regular, where $\h_0$ is a $\theta$-stable 
Cartan subalgebra of $\g_0$. Note that $\fc_1^c$ is spanned by $\h_0^c$-weight
vectors. So $\fc_1^c$ is normalised by $\h_0^c$, and therefore $\fc_1$
is normalised by $\h_0$. 
Let $H_0$ be the connected Lie subgroup of $G_0$ with Lie algebra $\h_0$;  we describe how to 'parametrise'
$H_0$ and find  elements in $\rho(H_0)$. Recall that $\h_0$ is a toral subalgebra of $\g$; we start with a general construction for such subalgebras.

Let $\ft^c$ be a toral subalgebra of $\g^c$, which we
assume to be algebraic, that is, there is an algebraic subgroup $T^c$ of $G^c$
whose Lie algebra is $\ad_{\g^c}(\ft^c)\cong \ft^c$; let $n= \dim\g^c$. First, suppose that $\ad_{\g^c} (\ft^c)$ consists of diagonal matrices.
Let
$\Lambda\subseteq \Z^n$ be the lattice defined by 
$$\Lambda = \{ (e_1,\ldots,e_n)\in \Z^n \mid e_1\alpha_1+\ldots+e_n\alpha_n = 
0 \text{ for all } \diag(\alpha_1,\ldots,\alpha_n)\in {\ad}_{\g^c}
(\ft^c)\};$$
here $\diag(\alpha_1,\ldots,\alpha_n)$ denotes the diagonal matrix
with diagonal entries $\alpha_1,\ldots,\alpha_n$. If we define
$$T^c = \{ \diag(\alpha_1,\ldots,\alpha_n)\in \C^n \mid \alpha_1^{e_1}\cdots\alpha_n^{e_n}=1 \text{ for all } e = (e_1,\ldots,e_n)\in \Lambda\},$$
then $T^c$ is a connected algebraic subgroup of $\GL(n,\C)$ with Lie algebra 
$\ad_{\g^c} (\ft^c)$, see \cite[\S II.7.3]{borel};
the set-up considered in \cite{borel} is slightly 
different, but the proof is the same; alternatively, see  \cite[\S 13, Prop.\ II.3]{chevii}. Let $E= \{ e^1,\ldots,e^r\}$ with $e^k=(e_1^k,\ldots,e_n^k)$ be a $\Z$-basis of $\Lambda$, and define $L\subseteq \Z^n$ as the lattice consisting of all $(d_1,\ldots,d_n)$ with
$d_1 e^k_1+\ldots+d_n e^k_n = 0$ for all $1\leq k\leq r$. Let $\{d^1,\ldots,d^s\}$ be a basis 
of $L$ and define \[\eta_j \colon \C^*\to T^c, \quad  t\mapsto \diag(t^{d_1^j},\ldots,t^{d_n^j}).\] Now the map $\eta\colon (\C^*)^s \to
T^c$,  $(t_1,\ldots,t_s) \mapsto \eta_1(t_1)\cdots \eta_s(t_s)$, is an isomorphism of algebraic groups. On the other hand, if $\ft^c$ is diagonalisable, but not diagonal, then there is  $A\in \GL(n,\C)$ such that $\widetilde{\ft}^c = A\ad_{\g^c} (\ft^c) A^{-1}$ 
consists of diagonal matrices. The construction above, applied to $\widetilde{\ft}^c$, yields an isomorphism $\widetilde{\eta}\colon (\C^*)^s \to 
\widetilde{T}^c$ and we define $\eta \colon (\C^*)^s \to T^c$,  $a\mapsto A^{-1} \widetilde{\eta}(a) A$. Thus, in both cases, we obtain an isomorphism \[\eta \colon (\C^*)^s \to T^c.\]
We apply  this procedure to $\ft^c=\h_0^c$ and $T^c=H_0^c\subset G_0^c$. Then $\eta((\C^*)^s) \cap \GL(n,\R) \subset
G_0$, and restricting elements of this set to $\g_1$ yields elements of 
$\rho(G_0)$, normalising $\fc_1$. 

It turns out that it is a good idea to apply this construction to the toral subalgebras $\h_0^c\cap\fk^c$ and $\h_0^c\cap\fp^c$ separately. For the latter algebra, the diagonalising matrix $A$
is defined over $\R$, hence we can simply restrict $\eta$ to $(\R^*)^s$ to get a large set of semisimple elements of $\rho(G_0)$. For $\h_0^c\cap\fk^c$ the diagonalising matrix $A$ is not defined over $\R$, and it may not be straightforward to find the subset of $(\C^*)^s$ which is mapped by $\eta$ into $G_0$. However, typically, the matrix $\eta_j(t)$ has a 
block diagonal form with $2\times 2$-blocks of the form
$B(t^k)$  for some $k\in\Z$, where
$$B(t) = \left(\begin{smallmatrix} \tfrac{1}{2}(t+t^{-1}) & \tfrac{1}{2}\imath (t-t^{-1})\\
-\tfrac{1}{2}\imath (t-t^{-1}) & \tfrac{1}{2}(t+t^{-1})\end{smallmatrix}\right),$$ see the comment below. Note that  $B(s)B(t) = B(st)$, and $B(t)$ has coefficients in $\R$ if and only if $t=x+\imath y$ with
$x^2+y^2=1$; in which case
$$B(t) = \left(\begin{smallmatrix} x & y \\ -y & x\end{smallmatrix}\right)$$
is the matrix describing rotation about the origin of $\R^2$ about the angle $\arccos(x)$. If $k\ne 0$, then  $t=x+\imath y$ runs through the circle defined by $x^2+y^2=1$ if and only
if $t^k$ does so. The nonzero orbits of
the group consisting of all $B(t)$, with $t=x+\imath y$ and $x^2+y^2=1$,
acting on $\R^2$, have representatives $(\alpha,0)$, $\alpha\geq 0$. 

In all examples we considered, the matrices of $\eta_j(t)$ had this
block diagonal form; this is no coincidence, in view of the normal form
theorem for orthogonal real matrices \cite[Thm VI.9]{jac2}, 
and the fact that the elements of the subgroup of $G$ corresponding to
$\fk$ are orthogonal transformations, cf.\ \cite[Prop.\ 5.1(i)]{onishchik}.

\subsection{$G_0$-invariants}\label{sec:invar} We describe some methods for establishing that two nilpotent
$e_1,e_2\in \g_1$ are {\em not} $G_0$-conjugate. A very powerful invariant is
the carrier algebra: if the carrier algebras of $e_1$ and $e_2$ are not isomorphic, then $e_1$ and $e_2$ cannot be $G_0$-conjugate. However, as we construct nilpotent elements in $\g_1$ 
by first listing the $G_0$-conjugacy classes of carrier algebras,
we may assume that both $e_1$ and $e_2$ have the same carrier algebra $\fc$. In particular,
viewed as elements of $\g_1^c$, they are $G_0^c$-conjugate. 
\begin{ithm}
\item[(i)] If $e_1,e_2$ are $G_0$-conjugate, then their centralisers $\fz_\g(e_i)$
are as well. Thus, if these centralisers contain different $G$-classes of 
nilpotent elements, then the $e_i$ cannot be $G_0$-conjugate. This 
can be checked using the Kostant-Sekiguchi correspondence, see \cite[\S 9.5]{colmcgov}, and \cite{dg} for a computational approach.
Similarly, we can study the centralisers $\fz_{\g_0}(e_i)$.
\item[(ii)] Let $(h_i,e_i,f_i)$ be a homogeneous $\SL_2$-triple, spanning the subalgebra $\fa_i$; by Remark \ref{remGP}, we have $h_1=h_2=h$, where  $h/2$ is the unique defining element of the carrier algebra of $e_1$ and $e_2$.  Recall that  $e_1$ and $e_2$ are $G_0$-conjugate if and only if these triples are; thus we can consider the centralisers
$\fz_\g(\fa_1)$ and $\fz_{\g}(\fa_2)$. We can check whether they are
isomorphic over $\R$, or whether they contain different $G$-classes of nilpotent elements. If any of these tests fails, then $e_1$ and $e_2$ cannot be $G_0$-conjugate. Similarly, we can study the centralisers $\fz_{\g_0}(\fa_i)$.
\item[(iii)] Using the notation of (ii), we also investigate centralisers
  of the semisimple elements $e_i-f_i$ or $e_i+f_i$. 
\end{ithm}

\section{Examples of nilpotent orbit computations}\label{secex}

\noindent We discuss in detail three interesting examples, motivated
by the literature; the gradings we use are constructed as in Example \ref{exa:grad}. In particular, this construction already gives as a Cartan involution which reverses the grading, which is supposed in most of our main algorithms.

\begin{example}
We consider a $\Z$-graded simple Lie algebra of type $E_8$, as in 
Example \ref{exa:grad}a). The degrees of the simple roots are given by
the diagram
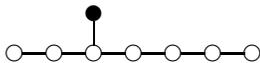
\begin{center}
\begin{picture}(70,17)
\put(3,0){\circle{6}}
  \put(18,0){\circle{6}}
  \put(33,0){\circle{6}}
  \put(48,0){\circle{6}}
  \put(33,15){\circle*{6}}
\put(63,0){\circle{6}}
\put(78,0){\circle{6}}
\put(93,0){\circle{6}}
  \put(6,0){\line(1,0){9}}
  \put(21,0){\line(1,0){9}}
  \put(36,0){\line(1,0){9}}
  \put(33,3){\line(0,1){9}}
  \put(51,0){\line(1,0){9}}
  \put(66,0){\line(1,0){9}}
  \put(81,0){\line(1,0){9}}
\end{picture}
\end{center}
where the simple root corresponding to the black node has degree 1, and all  others
have degree 0. This  example has also been considered by Djokovi{\'c} \cite{djoTri}, and   $G_0 \cong \GL(8,\R)$ and  $\g_1 \cong  \R^8\wedge \R^8\wedge \R^8$ as $G_0$-module. Djokovi{\'c} classified the nilpotent $G_0$-orbits in $\g_1$ using an approach based
on Galois cohomology. Concerning the results obtained with our methods, we remark the following:
\begin{ithm}
\bullit There are 53 $G_0$-conjugacy classes of real carrier algebras
(whereas there are 22 $G_0^c$-conjugacy classes of carrier algebras in
$\g^c$); our program needed 2567 seconds to list them. 
\bullit Although there are real carrier algebras which contain representatives of more than one nilpotent $G_0$-orbit, in all
cases all (but at most one) are ruled out by the condition of Proposition
\ref{propcarrk}c). So each real carrier algebra corresponds to at most one 
nilpotent orbit. In particular, in this example, there is no need
to use the criteria outlined in Section \ref{sec:invar}.
\bullit Exactly 34 carrier algebras correspond to exactly one nilpotent orbit, the
others correspond to no nilpotent orbit.
\bullit We find that each complex orbit splits in exactly the same number
of real nilpotent orbits as in \cite{djoTri}; so our classification
and the one of Djokovi{\'c} are equivalent. Moreover, we checked the representatives
given by Djokovi{\'c}, and they all turned out to lie in the correct nilpotent orbit.
\bullit  Djokovi{\'c} also computed the isomorphism types of the centralisers
$\fz_{\g_0}(\fa)$, where $\fa$ is the subalgebra spanned by
a homogeneous $\SL_2$-triple. Here, on some occasions, we get different results. For example, the complex orbit labelled ``VI'' in \cite{djoTri} has only one real
representative, and Djokovi{\'c} claims  that  $\fz_{\g_0}(\fa)$ is isomorphic to $3\SL_2(\R)$;  we find that  $\fz_{\g_0}(\fa)$ is isomorphic to $\mathfrak{sp}_6(\R) \oplus \ft_2$, where $\ft_2$ is a $2$-dimensional toral subalgebra lying 
inside $\fp$.
\end{ithm} 
\end{example}

\begin{example}
We consider a $\Z_2$-graded simple Lie algebra of type $G_2$, as in 
Example \ref{exa:grad}b). The complex simple Lie algebra of type $G_2$ has
only one (conjugacy class of) involution, which we use to construct
the grading; this grading is also studied in the physics literature,
see \cite{G2phys}. Here $\g_0$ has four (classes of) Cartan
subalgebras, but strongly regular carrier algebras only exist with
respect to the split Cartan subalgebra. There are five of these carrier algebras, like in the complex case. Three of them yield one nilpotent orbit, and
two correspond to two nilpotent orbits. This leads to an equivalent 
classification as derived in \cite{G2phys}.

We now study the carrier algebra $\fc$ that is isomorphic to $\g$ (as Lie algebra) in a more detail:  it has dimension 14, and $\dim \fc_0 = \dim \fc_1= 4$. Let $\{y_1,\ldots,y_4\}$ be  a fixed basis of $\fc_1$, as computed by our
programs. Using the approach of Section \ref{sec:genpos},  an element $\sum_i c_i y_i\in\fc_1$ is in general position if and
only if   
\begin{equation}\label{eqdense}
c_1^2c_4^2-6c_1c_2c_3c_4+4c_1c_3^3+4c_2^3c_4-3c_2^2c_3^2\neq 0.
\end{equation}
Unfortunately, this polynomial is irreducible over $\Q$. The semisimple part of
$\fc_0$ is isomorphic to $\SL_2(\R)$, so it has two nilpotent basis elements,  denoted by $u$ and $v$.
The exponentials of their adjoints, restricted to $\fc_1$, are 
$$\exp( s \ad\nolimits_\g(u) ) = \left(\begin{smallmatrix}
1 &       0 &       0 &       0 \\
        -s &       1 &       0 &       0 \\
   s^2 &  -2s &       1 &       0  \\
  -s^3 &  3s^2 &  -3s &       1 \end{smallmatrix}\right)\quad\text{and}\quad \exp( t \ad\nolimits_\g(v) ) = \left(\begin{smallmatrix}
1 &  -3t &  3t^2 &    -t^3 \\
         0 &       1 &  -2t &   t^2 \\
         0 &       0 &       1 &      -t \\
         0 &       0 &       0 &       1\end{smallmatrix}\right),$$
{where $s,t\in\R$.} Let $e=\sum_i c_i y_i$ be an element in general position; we now use Section \ref{secsn} and act with $\exp(s\ad_\g(u))$ and  $\exp(t\ad_\g(v))$. 
First, after
acting with an $\exp( s \ad (u) )$ (if required), we can assume $c_4\neq 0$. 
Note that  $\exp(t \ad_\g (v)) e = \sum_i c_i' y_i$ where $c_1' = c_1-3tc_2+3t^2c_3-t^3
c_4$ is a polynomial in $t$ of degree 3; this polynomial has a real zero $t_0$, and, by acting with $\exp(t_0 \ad_\g (v))$, we may assume that $c_1=0$. Now it follows from \eqref{eqdense} that $c_2\neq 0$. Write $\exp(s \ad_\g (u)) e=
\sum_i c_i' y_i$, so that $c_1'=0$, $c_2'=c_2$ and $c_3' = -2sc_2+c_3$. We can choose $s$ so that  $c_3'=0$; in conclusion,  we may assume that $c_1=c_3=0$, and $c_2\ne 0\ne c_4$ by  \eqref{eqdense}.

Using Section \ref{secsss}, the group corresponding to the split Cartan subalgebra of $\g_0$ acts on $\fc_1$ as
$$\diag(a_1^3a_2^8,a_1^2a_2^5,a_1a_2^2,a_2^{-1}).$$
where each $a_i\in \R^*$. Acting with this group on $e=c_2y_2+c_4y_4$, we can multiply $c_4$ by an arbitrary $b\in \R^*$,
and $c_2$ by an $a\in \R^*$, provided that $ab$ is a square in $\R^*$.
We get two possible nilpotent orbits,  represented by $e_1 = y_2+y_4$ and  $e_2=-y_2+y_4$ respectively. For $i=1,2$ let $\fa_i$ denote the subalgebra
spanned by the homogeneous $\SL_2$-triple $(h,e_i,f_i)$. It turns out that
$\fz_{\g_0}(\fa_i) = 0$ and $\fz_{\g_0}(\fc)=0$, so the condition of Proposition
\ref{propcarrk}c) is trivially satisfied. Now consider $u_i = e_i-f_i$; its centraliser in $\g_0$ is $1$-dimensional and spanned by a semisimple
element $t_i$. The minimal polynomials of $\ad_{\g_0}(t_1)$ and  $\ad_{\g_0}(t_2)$ are $(X-6)(X-2)X(X+2)(X+6)$ and $X(X^2+4)(X^2+36)$, respectively. So $t_1$ is
not $G_0$-conjugate to $\lambda t_2$ for all $\lambda\in\R$. Thus, the 
centralisers of the $u_i$ are not $G_0$-conjugate, and hence neither are
$e_1$ and $e_2$.
\end{example}

\begin{example}
We consider a $\Z$-graded simple Lie algebra of type $E_8$, as in 
Example \ref{exa:grad}a). The degrees of the simple roots are given by
the diagram
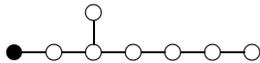
\begin{center}
\begin{picture}(70,17)
\put(3,0){\circle*{6}}
  \put(18,0){\circle{6}}
  \put(33,0){\circle{6}}
  \put(48,0){\circle{6}}
  \put(33,15){\circle{6}}
\put(63,0){\circle{6}}
\put(78,0){\circle{6}}
\put(93,0){\circle{6}}
  \put(6,0){\line(1,0){9}}
  \put(21,0){\line(1,0){9}}
  \put(36,0){\line(1,0){9}}
  \put(33,3){\line(0,1){9}}
  \put(51,0){\line(1,0){9}}
  \put(66,0){\line(1,0){9}}
  \put(81,0){\line(1,0){9}}
\end{picture}
\end{center}
where the simple root corresponding to the black node has degree 1, 
and all  others have degree 0. Over the complex numbers this grading has been
considered in \cite{gattivin}, where it is shown that there are nine
nonzero orbits, and for each orbit a representative is given. 
Here $G_0^c \cong \mathrm{Spin}_{14}(\C)\times \C^*$ and $\g_1^c$ is (as 
$G_0^c$-module) the 64-dimensional spinor representation of
$\mathrm{Spin}_{14}(\C)$. (The orbits of $\mathrm{Spin}_{14}(\C)$
on this space have also been classified by Popov in \cite{popov}; but without
making use of graded Lie algebras.)
We have computed the orbits of $G_0$ acting on $\g_1$, and we remark the
following:
\begin{items}
\bullit It took 4409 seconds to complete the classification of the carrier
algebras.
\bullit There are 10 Cartan subalgebras in $\g_0$ (up to $G_0$-conjugacy).
\bullit There are 26 carrier algebras (up to $G_0$-conjugacy), and all of
them are principal. 
\bullit Each carrier algebra corresponds to at most one orbit; this yields
15 orbits in total.
\end{items}

In Table \ref{tab:spin} we summarise the results of our computations. 
In this table, the first column lists the label of the corresponding
complex orbit; we use the same numbering as in \cite{gattivin}. In
particular, the rows of the table
correspond to the complex orbits. The second column contains the 
dimension of the orbit. The third column lists a representative of the orbit.
These representatives are given as spinors; for the notation we refer to
\cite[\S 2.0]{gattivin}. If a complex orbit splits into more than one real orbit, then the third and 
fourth column have extra lines in the corresponding row, with the
corresponding information for each real orbit. The fourth column
contains the isomorphism type of the centraliser in $\g_0$ of a homogeneous
$\SL_2$-triple containing the given representative. We use the
following notation: $\ft_{r,s}$ denotes a $\theta$-stable toral subalgebra
such that $\dim(\fk\cap \ft_{r,s})=r$ and $\dim(\fp\cap \ft_{r,s})=s$
where $\g=\fk\oplus\fp$ as usual; the
Lie algebras $G_2$, $G_2(\C)$, and $G_2^{\textrm{cpt}}$ are the split real 
form of the Lie algebra if type $G_2$, the complex Lie algebra of type $G_2$ seen as 
real, and the compact real form of $G_2$, respectively.

\begin{table}[ht]
\caption{Real orbits of $\mathrm{Spin}_{14}(\R)\times \R^*$}\label{tab:spin}
\setlength{\extrarowheight}{0.6ex}
\centering \small
\begin{tabular}{|c|c|l|c|}
\hline
\scc orbit & \scc $\dim$ &\multicolumn{1}{|c|}{\scc representative} &\scc $\fz_{\g_0}(\fa)$ \\
\hline
I & 22 & $f_1f_5f_6f_7$ & $\SL_7(\R)\oplus \ft_{0,1}$\\
\hline
II & 35 & $f_1f_5f_6f_7-f_2f_4f_5f_6$ & $\SL_3(\R)\oplus \mathfrak{so}_7(\R)\oplus
\ft_{0,1}$\\
\hline
III & 43 & $f_1f_2+f_1f_4f_6f_7+f_2f_3f_4f_5$ & $\mathfrak{sp}_6(\R)\oplus
\ft_{0,2}$\\
\hline 
IV & 44 & $f_1f_5f_6f_7+f_2f_3f_5f_7$ & $\SL_6(\R)\oplus \ft_{0,1}$\\\cdashline{3-4}
& & $-f_1f_2f_4f_7+f_1f_2f_5f_6-f_1f_2f_5f_7-f_1f_4f_5f_6$ &  $\mathfrak{su}(3,3)
\oplus \ft_{0,1}$\\
\hline
V & 50 & $f_1f_2+f_1f_4f_6f_7+f_1f_5f_6f_7+f_2f_3f_4f_5$ & $\SL_4(\R)\oplus 
\ft_{0,1}$\\
\hline 
VI & 54 & $f_1f_3+f_1f_5f_6f_7-f_2f_4f_5f_6$ & $2\SL_3(\R)\oplus 
\ft_{0,1}$\\\cdashline{3-4}
&& $-f_1f_2f_4f_7-f_1f_2f_5f_6-f_1f_2f_5f_7-f_1f_3f_4f_5+f_1f_4f_5f_6$ &
$\SL_3(\C)\oplus \ft_{0,1}$\\
\hline
VII & 59 & $f_1f_2+f_1f_4f_6f_7+f_2f_3f_4f_5-f_3f_4f_5f_6$ & 
$\SL_2(\R)\oplus \mathfrak{so}_5(\R)\oplus \ft_{0,1}$\\\cdashline{3-4}
&& $-f_1f_2f_4f_7+f_1f_2f_5f_6-f_1f_2f_5f_7-f_1f_4f_5f_6-f_2f_3f_6f_7+$
 & $\mathfrak{su}(2)\oplus\mathfrak{so}(1,4)\oplus \ft_{0,1}$\\
&& $f_1f_2f_3f_4f_6f_7$ & \\
\hline
VIII & 63 & $f_1f_2+f_1f_4f_6f_7+f_1f_5f_6f_7+f_2f_3f_4f_5-f_3f_4f_5f_6$ &
$G_2\oplus \ft_{0,1}$\\\cdashline{3-4}
&& $1-f_1f_2f_4f_7+f_1f_2f_5f_6+f_1f_2f_5f_7+f_1f_4f_5f_6+f_2f_3f_6f_7+$
& $G_2^{\mathrm{cpt}}\oplus \ft_{0,1}$\\
&& $f_1f_2f_3f_4f_6f_7$ & \\
\hline
IX & 64 & $f_1f_3+f_1f_5f_6f_7+f_2f_3f_5f_7-f_2f_4f_5f_6$ & $2G_2$ \\\cdashline{3-4}
&& $f_4f_5-f_1f_2f_4f_7-f_1f_2f_5f_6+f_1f_2f_5f_7-f_1f_4f_5f_6+f_2f_3f_6f_7$ &
$G_2(\C)$\\\cdashline{3-4}
&& $f_1f_2-f_5f_6+f_1f_2f_4f_5+f_1f_2f_4f_6-f_1f_4f_5f_7-f_1f_5f_6f_7+$
& $2G_2^{\mathrm{cpt}}$ \\
&& $f_2f_3f_4f_5-f_1f_2f_4f_5f_6f_7$ & \\
\hline
\end{tabular}
\end{table}
\end{example}

\bibliographystyle{plain}

\begin{thebibliography}{10}

\bibitem{atlasLG}
Atlas of {L}ie groups and representations.
\newblock see \url{liegroups.org}.

\bibitem{adams}
Jeffrey Adams.
\newblock Galois and $\theta$ cohomology of real groups.
\newblock arXiv:1310.7917 [math.GR], 2014.

\bibitem{adclou}
Jeffrey Adams and Fokko du~Cloux.
\newblock Algorithms for representation theory of real reductive groups.
\newblock {\em J.\ Inst.\ Math.\ Jussieu}, 8(2):209--259, 2009.

\bibitem{antelash}
L.~V.\ Antonyan and A.~G.\ {\`E}lashvili.
\newblock Classification of spinors of dimension sixteen.
\newblock {\em Trudy Tbiliss.\ Mat.\ Inst.\ Razmadze Akad.\ Nauk Gruzin. SSR},
  70:5--23, 1982.

\bibitem{borel}
A.~Borel.
\newblock {\em Linear algebraic groups}.
\newblock Springer-Verlag, Berlin, Heidelberg, New York, second edition, 1991.

\bibitem{bou}
N.~Bourbaki.
\newblock {\em {Groupes et Alg\`{e}bres de Lie, Chapitre I}}.
\newblock Hermann, Paris, 1971.

\bibitem{F4phys}
W.~Chemissany, P.~Giaccone, D.~Ruggeri, and M.~Trigiante.
\newblock Black hole solutions to the {$F_4$}-model and their orbits ({I}).
\newblock {\em Nuclear Phys.\ B}, 863(1):260--328, 2012.

\bibitem{chevii}
Claude Chevalley.
\newblock {\em Th\'eorie des groupes de {L}ie. {T}ome {II}. {G}roupes
  alg\'ebriques}.
\newblock Actualit\'es Sci.\ Ind.\ no.\ 1152. Hermann \& Cie., Paris, 1951.

\bibitem{colmcgov}
David~H.\ Collingwood and William~M.\ McGovern.
\newblock {\em Nilpotent orbits in semisimple {L}ie algebras}.
\newblock Van Nostrand Reinhold Mathematics Series. Van Nostrand Reinhold Co.,
  New York, 1993.

\bibitem{cornsub1}
J.~F.\ Cornwell.
\newblock Semi-simple real subalgebras of non-compact semi-simple real {L}ie
  algebras. {I}, {II}.
\newblock {\em Rep.\ Mathematical Phys.}, 2(4):239--261; ibid. 2 (1971), no.\ 4,
  289--309, 1971.

\bibitem{cornsub2}
J.~F.\ Cornwell.
\newblock Semi-simple real subalgebras of non-compact semi-simple real {L}ie
  algebras. {III}.
\newblock {\em Rep.\ Mathematical Phys.}, 3(2):91--107, 1972.

\bibitem{dGraafSS}
Willem~A.\ de~Graaf.
\newblock Constructing semisimple subalgebras of semisimple {L}ie algebras.
\newblock {\em J.\ Algebra}, 325:416--430, 2011.

\bibitem{corelg}
Heiko Dietrich, Paolo Faccin, and Willem A.~de Graaf.
\newblock Corelg: Computing with real {L}ie groups.
\newblock A GAP 4 package, available at \url{science.unitn.it/~corelg/}.

\bibitem{dfg}
Heiko Dietrich, Paolo Faccin, and Willem A.~de Graaf.
\newblock Computing with real {L}ie algebras: real forms, {C}artan
  decompositions, and {C}artan subalgebras.
\newblock {\em J.\ Symbolic Comput.}, 56:27--45, 2013.

\bibitem{dg}
Heiko Dietrich and Willem A.~de Graaf.
\newblock A computational approach to the {K}ostant-{S}ekiguchi correspondence.
\newblock {\em Pacific Journal of Mathematics}, 265(2):349--379, 2013.

\bibitem{djoZ}
Dragomir~{\v{Z}}.\ Djokovi{\'c}.
\newblock Classification of {${\bf Z}$}-graded real semisimple {L}ie algebras.
\newblock {\em J.\ Algebra}, 76(2):367--382, 1982.

\bibitem{djoTri}
Dragomir~{\v{Z}}.\ Djokovic.
\newblock Classification of trivectors of an eight-dimensional real vector
  space.
\newblock {\em Linear and Multilinear Algebra}, 13(1):3--39, 1983.

\bibitem{dyn0}
E.~B.\ Dynkin.
\newblock Maximal subgroups of the classical groups.
\newblock {\em Trudy Moskov.\ Mat.\ Ob\v s\v c.}, 1:39--166, 1952.
\newblock English translation in: Amer.\ Math.\ Soc.\ Transl.\ (6), (1957),
  245--378.

\bibitem{dyn}
E.~B.\ Dynkin.
\newblock Semisimple subalgebras of semisimple {L}ie algebras.
\newblock {\em Mat.\ Sbornik N.S.}, 30(72):349--462 (3 plates), 1952.
\newblock English translation in: Amer.\ Math.\ Soc.\ Transl.\ (6), (1957),
  111--244.

\bibitem{cornsub3}
J.~M.\ Ekins and J.~F.\ Cornwell.
\newblock Semi-simple real subalgebras of non-compact semi-simple real {L}ie
  algebras. {IV}.
\newblock {\em Rep.\ Mathematical Phys.}, 5(1):17--49, 1974.

\bibitem{cornsub4}
J.~M.\ Ekins and J.~F.\ Cornwell.
\newblock Semi-simple real subalgebras of non-compact semi-simple real {L}ie
  algebras. {V}.
\newblock {\em Rep.\ Mathematical Phys.}, 7(2):167--203, 1975.

\bibitem{elavin}
A.\~G.\ \'Ela\v{s}vili and {\`E}.~B.\ Vinberg.
\newblock A classification of the 3-vectors of 9-dimensional space.
\newblock {\em Trudy Sem.\ Vektor.\ Tenzor.\ Anal.}, (18):197--233, 1978.
\newblock English translation: Selecta Math.\ Sov.\ 7, 63-98 (1988).

\bibitem{grafac}
Paolo Faccin and Willem A.~de Graaf.
\newblock Constructing semisimple subalgebras of real semisimple {L}ie
  algebras.
\newblock In {\em Proceedings of the Bicocca-Workshop on Lie Algebras 2013}.
\newblock to appear.

\bibitem{gatim}
L.~Yu.\ Galitski and D.~A.\ Timashev.
\newblock On classification of metabelian {L}ie algebras.
\newblock {\em J.\ of Lie Theory}, 9:125--156, 1999.

\bibitem{gattivin}
V.~Gatti and E.~Viniberghi.
\newblock Spinors of 13-dimensional space.
\newblock {\em Adv.\ Math.}, 30:137--155, 1978.

\bibitem{deGraafBook}
Willem A.~de Graaf.
\newblock {\em Lie Algebras: theory and algorithms}, volume~56 of {\em
  North-Holland Math.\ Lib.}
\newblock Elsevier, 2000.

\bibitem{gra15}
Willem A.~de Graaf.
\newblock Computing representatives of nilpotent orbits of $\theta$-groups.
\newblock {\em J.\ Symbolic Comput.}, 46:438--458, 2011.

\bibitem{gap}
The~GAP Group.
\newblock G{A}{P} -- groups, algorithms, and programming.\ v.4.7.2.
\newblock Available at \url{gap-system.org}.

\bibitem{helgasson}
Sigurdur Helgason.
\newblock {\em Differential geometry, {L}ie groups, and symmetric spaces},
  volume~80 of {\em Pure and Applied Mathematics}.
\newblock Academic Press Inc.\ [Harcourt Brace Jovanovich Publishers], New York,
  1978.

\bibitem{hitchin}
Nigel Hitchin.
\newblock Stable forms and special metrics.
\newblock In {\em Global differential geometry: the mathematical legacy of
  {A}lfred {G}ray ({B}ilbao, 2000)}, volume 288 of {\em Contemp.\ Math.}, pages
  70--89. Amer.\ Math.\ Soc., Providence, RI, 2001.

\bibitem{hum}
J.~E.\ Humphreys.
\newblock {\em {Introduction to Lie Algebras and Representation Theory}}.
\newblock Springer Verlag, New York, Heidelberg, Berlin, 1972.

\bibitem{jac}
N.~Jacobson.
\newblock {\em {Lie Algebras}}.
\newblock Dover, New York, 1979.

\bibitem{jac2}
Nathan Jacobson.
\newblock {\em Lectures in abstract algebra}.
\newblock Springer-Verlag, New York, 1975.
\newblock Volume II: Linear algebra, Reprint of the 1953 edition [Van Nostrand,
  Toronto, Ont.], Graduate Texts in Mathematics, No.\ 31.

\bibitem{G2phys}
Sung-Soo Kim, Josef Lindman~H{\"o}rnlund, Jakob Palmkvist, and Amitabh Virmani.
\newblock Extremal solutions of the {$S^3$} model and nilpotent orbits of
  {$G_{2(2)}$}.
\newblock {\em J.\ High Energy Phys.}, (8):072, 43, 2010.

\bibitem{knapp02}
A.~W.\ Knapp.
\newblock {\em Lie groups beyond an introduction}, volume 140 of {\em Progress
  in Mathematics}.
\newblock Birkh\"auser Boston Inc., Boston, MA, second edition, 2002.

\bibitem{komrakov}
B.~P.\ Komrakov.
\newblock Maximal subalgebras of real {L}ie algebras and a problem of {S}ophus
  {L}ie.
\newblock {\em Dokl.\ Akad.\ Nauk SSSR}, 311(3):528--532, 1990.

\bibitem{kostant}
Bertram Kostant.
\newblock On the conjugacy of real {C}artan subalgebras. {I}.
\newblock {\em Proc.\ Nat.\ Acad.\ Sci.\ U.\ S.\ A.}, 41:967--970, 1955.

\bibitem{hongvanle}
H{\^o}ng~V{\^a}n L{\^e}.
\newblock Orbits in real {$Z_m$}-graded semisimple {L}ie algebras.
\newblock {\em J.\ Lie Theory}, 21(2):285--305, 2011.

\bibitem{litt8}
P.~Littelmann.
\newblock An effective method to classify nilpotent orbits.
\newblock In {\em Algorithms in algebraic geometry and applications
  ({S}antander, 1994)}, volume 143 of {\em Progr.\ Math.}, pages 255--269.
  Birkh\"auser, Basel, 1996.

\bibitem{logru}
M.~Lorente and B.~Gruber.
\newblock Classification of semisimple subalgebras of simple {L}ie algebras.
\newblock {\em J.\ Mathematical Phys.}, 13:1639--1663, 1972.

\bibitem{minchenko}
A.~N.\ Minchenko.
\newblock Semisimple subalgebras of exceptional {L}ie algebras.
\newblock {\em Tr.\ Mosk.\ Mat.\ Obs.}, 67:256--293, 2006.
\newblock English translation in: Trans.\ Moscow Math.\ Soc.\ 2006, 225--259.

\bibitem{onishchik}
Arkady~L.\ Onishchik.
\newblock {\em Lectures on Real Semisimple {L}ie Algebras and Their
  Representations}.
\newblock European Mathematical Society, Z\"urich, 2004.

\bibitem{popov}
V.~L.\ Popov.
\newblock A classification of spinors of dimension fourteen.
\newblock {\em Trudy Moskov.\ Mat.\ Obshch.}, 37:173--217, 270, 1978.

\bibitem{sims}
C.~C.\ Sims. 
\newblock {\em Computation with Finitely Presented Groups}.
\newblock Cambridge University Press, Cambridge, 1994.

\bibitem{sugiura}
Mitsuo Sugiura.
\newblock Conjugate classes of {C}artan subalgebras in real semi-simple {L}ie
  algebras.
\newblock {\em J.\ Math.\ Soc.\ Japan}, 11:374--434, 1959.

\bibitem{tauvelyu}
Patrice Tauvel and Rupert W.~T.\ Yu.
\newblock {\em Lie Algebras and Algebraic Groups}.
\newblock Springer-Verlag, Berlin Heidelberg New York, 2005.

\bibitem{vinberg3}
{\`E}.~B.\ Vinberg.
\newblock The classification of nilpotent elements of graded {L}ie algebras.
\newblock {\em Dokl.\ Akad.\ Nauk SSSR}, 225(4):745–748, 1975.

\bibitem{vinberg}
{\`E}.~B.\ Vinberg.
\newblock The {W}eyl group of a graded {L}ie algebra.
\newblock {\em Izv.\ Akad.\ Nauk SSSR Ser.\ Mat.}, 40(3):488--526, 1976.
\newblock English translation: Math.\ USSR-Izv.\ 10, 463--495 (1976).

\bibitem{vinberg2}
{\`E}.~B.\ Vinberg.
\newblock Classification of homogeneous nilpotent elements of a semisimple
  graded {L}ie algebra.
\newblock {\em Trudy Sem.\ Vektor.\ Tenzor.\ Anal.}, (19):155--177, 1979.
\newblock English translation: Selecta Math.\ Sov.\ 6, 15-35 (1987).

\bibitem{vogan}
David~A.\ Vogan, Jr.
\newblock Irreducible characters of semisimple {L}ie groups.\ {IV}.
  {C}haracter-multiplicity duality.
\newblock {\em Duke Math.\ J.}, 49(4):943--1073, 1982.

\bibitem{warner}
Garth Warner.
\newblock {\em Harmonic analysis on semi-simple {L}ie groups.\ {I}}.
\newblock Springer-Verlag, New York, 1972.
\newblock Die Grundlehren der mathematischen Wissenschaften, Band 188.

\end{thebibliography}

\end{document}